\documentclass[11pt]{article}

\usepackage[a4paper,hmargin={0.75in,0.75in},vmargin={0.6in,0.9in}]{geometry}

\usepackage[utf8]{inputenc}
\usepackage[english]{babel}
\usepackage{enumitem}
\usepackage{ifthen}
\usepackage{subcaption}

\usepackage[colorlinks=true,citecolor=cyan,linkcolor=magenta]{hyperref}
\usepackage{amsmath}
\usepackage{amsfonts}
\usepackage{amssymb}
\usepackage{amsthm}
\usepackage{thmtools}
\usepackage[capitalize,nameinlink]{cleveref} 
\crefname{equation}{Equation}{Equations}
\crefname{figure}{Figure}{Figures}
\newtheorem{theorem}{Theorem}[section]
\newtheorem{proposition}[theorem]{Proposition}
\newtheorem{lemma}[theorem]{Lemma}
\newtheorem{corollary}[theorem]{Corollary}

\usepackage{float}
\usepackage{tikz}
\usetikzlibrary{math}
\usetikzlibrary{calc}
\usetikzlibrary{positioning}

\tikzstyle{vtx}=[circle, draw, fill=black, inner sep=0pt, minimum width=5pt]
\tikzstyle{vtx-white}=[circle, draw, fill=white, inner sep=0pt, minimum width=5pt]

\newenvironment{eqnqed}{\begin{center} \hfill $\displaystyle}{$ \qed \end{center}}

\newcommand{\fig}[3]{
	\begin{figure}
		\centering
		#3
		\caption{#1}
		\label{#2}
	\end{figure}
}
\usepackage{xparse}
\ExplSyntaxOn
\NewDocumentCommand{\subfig}{m m m m}{
	\seq_set_split:Nnn \l_width_seq { , } { #1 }
	\seq_set_split:Nnn \l_caption_seq { , } { #2 }
	\seq_set_split:Nnn \l_label_seq { , } { #3 }
	\seq_set_split:Nnn \l_tikzfile_seq { , } { #4 }

	\int_step_inline:nn { \seq_count:N \l_width_seq }{
		\begin{subfigure}{\seq_item:Nn \l_width_seq { ##1 }\textwidth}
			\centering
			\input{\seq_item:Nn \l_tikzfile_seq { ##1 }}
			\caption{\seq_item:Nn \l_caption_seq { ##1 }}
			\label{\seq_item:Nn \l_label_seq { ##1 }}
		\end{subfigure}%
	}
}
\ExplSyntaxOff

\newcommand{\ds}{\displaystyle}
\newcommand{\Xbar}{\overline{X}}
\newcommand{\xbar}{\overline{x}}
\newcommand{\cupdot}{\mathbin{\mathaccent\cdot\cup}}

\makeatletter
\def\moverlay{\mathpalette\mov@rlay}
\def\mov@rlay#1#2{\leavevmode\vtop{%
	\baselineskip\z@skip \lineskiplimit-\maxdimen
	\ialign{\hfil$\m@th#1##$\hfil\cr#2\crcr}}}
\newcommand{\charfusion}[3][\mathord]{
		#1{\ifx#1\mathop\vphantom{#2}\fi
				\mathpalette\mov@rlay{#2\cr#3}
			}
		\ifx#1\mathop\expandafter\displaylimits\fi}
\makeatother
\newcommand{\bigcupdot}{\charfusion[\mathop]{\bigcup}{\cdot}}

\newcommand{\con}[1]{%
	\def\temp{#1}%
	\def\zero{0}%
	\ifx\temp\empty
		connected%
	\else
		\ifx\temp\zero
			disconnected%
		\else
			\mbox{$#1$-connected}%
		\fi
	\fi
}
\newcommand{\bcc}[1]{\mbox{$#1$-connected} bipartite cubic graph}
\newcommand{\cc}[1]{\mbox{$#1$-connected} cubic graph}

\newcommand{\mc}{matching covered}
\newcommand{\mcg}{\mc~graph}
\newcommand{\bmcg}{bipartite \mcg}
\newcommand{\codd}{\mathsf{c_{odd}}}

\newcommand{\f}{\pi}

\newcommand{\cfour}{\mathbb{C}_4}
\newcommand{\B}{\beta}
\renewcommand{\b}{\B'}

\newcommand{\lm}{\mbox{$\lambda$-matchable}}
\newcommand{\lmity}{\mbox{$\lambda$-matchability}}
\newcommand{\G}{\mathcal{G}}
\newcommand{\N}{\mathcal{N}}

\newcommand{\Rho}{\mathrm{P}}
\newcommand{\numlmpartners}[2]{\rho(#1,#2)}
\newcommand{\K}{\mathcal{K}}
\renewcommand{\L}{\mathcal{L}'}
\newcommand{\braces}{\mathcal{J}}

\newcommand{\thetabar}{\overline{\theta}}
\newcommand{\bstar}{\mathcal{B}^*}
\newcommand{\paritylemma}{the \hyperlink{parity lemma}{Parity Lemma (1.2)}}

\title{\lmity\ in cubic graphs}
\author{Santhosh Raghul\thanks{Supported by Ministry of Education, Government of India} \and Nishad Kothari\thanks{Supported by IC\&SR and aRtCS, IIT Madras}}

\date{IIT Madras \\~\\ \vspace{-1mm}Dedicated to beloved Professor Murty}

\begin{document}

\maketitle
\thispagestyle{empty}

\vspace{-8mm}
\begin{abstract}
	A vertex $v$ of a $2$-connected cubic graph $G$ is \emph{$\lambda$-matchable} if $G$ has a spanning subgraph in which $v$ has degree three whereas every other vertex has degree one, and we let $\lambda(G)$ denote the number of such vertices.
Clearly, $\lambda=0$ for bipartite graphs; ergo, we define \emph{$\lambda$-matchable pairs} analogously, and we let $\rho(G)$ denote the number of such pairs.

We improve the constant lower bounds on both $\lambda$ and $\rho$ established recently by Chen, Lu and Zhang [Discrete Math., 2025] using matching-theoretic invariants arising from the seminal work of Lov\'asz [J. Combin. Theory Ser. B, 1987], and we characterize all of the tight examples.
We also solve the problem posed by Chen, Lu and Zhang: characterize $2$-connected cubic graphs each of whose vertices is $\lambda$-matchable.

\end{abstract}


\section{Cubic graphs and \lmity}\label{sec: cubic graphs and lmity}
For graph-theoretic terminology, we follow Bondy and Murty \cite{bomu08}, whereas for terminology specific to matching theory, we follow Lucchesi and Murty \cite{lumu24}.
We adopt notational conventions from both; for instance, for sets $S$ and $T$, we use $S-T$ to denote $S \setminus T$, and if $T:=\{t\}$, we simplify $S - \{t\}$ to~$S-t$.
All graphs herein are loopless; however, we allow multiple/parallel edges.
By \emph{order} of a graph $G$, denoted by $n(G)$ or simply $n$, we mean its number of vertices; 
we use $\kappa(G)$ and $\kappa'(G)$ to denote its vertex-connectivity and edge-connectivity, respectively.
The following fact about $3$-regular (aka \emph{cubic}) graphs is well-known and easily proved.

\begin{proposition}\label{for cubic vtx-conn = edge-conn}
	Every cubic graph $G$ satisfies $\kappa(G) = \kappa'(G)$. \qed
\end{proposition}

A graph is \emph{matchable} if it has a perfect matching, and an edge is \emph{matchable} if it belongs to some perfect matching.
A connected graph, of order two or more, is \emph{matching covered} if each of its edges is matchable.
Sch{\"o}nberger (1934) proved that a cubic graph is matching covered if and only if it is \con2.
\cref{fig: smallest 2-conn cubic graphs} shows the smallest \cc2s along with the corresponding notation.

\fig{A few small \cc2s}{fig: smallest 2-conn cubic graphs}{
	\subfig
	{ 0.15 , 0.17 , 0.17 , 0.25 , 0.25 } 
	{ $\Theta$ , $\cfour$ , $K_4$ , $K_{3,3}$ , $\overline{C_6}$ } 
	{fig: theta,fig: cubicc4,fig: k4,fig: k33,fig: c6bar} 
	{ theta.tikz , cubicc4.tikz , k4.tikz , k33.tikz , c6bar.tikz } 
}

Clearly, one may view a spanning subgraph as a set of edges and vice versa.
In this spirit, a perfect matching is a $1$-regular spanning subgraph.
We now proceed to define the notion of our interest.
For a vertex $v$ of a \cc2 $G$, a \emph{\mbox{$\lambda$-matching} centered at $v$}, or simply a \emph{\mbox{$v$-matching}}, is any spanning subgraph in which $v$ has degree three whereas every other vertex has degree one;
furthermore, we say that $v$ is \emph{\lm} if $G$ admits a \mbox{$v$-matching}.
Ergo, both perfect matchings and \mbox{$\lambda$-matchings} are spanning subgraphs with the property that each vertex has odd degree;
next, we intend to state a pertinent observation that applies to all such graphs.

To this end, we require the notion of cuts.
For a graph~$G$ and a subset $X \subseteq V(G)$, the set comprising those edges that have one end in $X$ and the other end in its complement $\Xbar:=V(G)-X$, is called a \emph{cut} of~$G$,
and is denoted by $\partial_G(X)$ or simply by $\partial(X)$;
we refer to $X$~and~$\Xbar$ as its \emph{shores}.
A cut $C$ is a \emph{\mbox{$k$-cut}} if $|C|=k$.
We are now ready to state the aforementioned observation in a more general context; it is an immediate consequence of the handshaking lemma, and shall be invoked throughout this paper.

\begin{lemma}
	\hypertarget{parity lemma}{{\sc [Parity Lemma]}}
	\label{parity lemma}
	\newline
	For a graph $G$, if $X$ is any subset of $V(G)$ whose each member has odd degree, then:
	\begin{eqnqed}
		|\partial(X)| ~\equiv~ |X|~(\rm{mod}~2)
	\end{eqnqed}
\end{lemma}

In their recent work, Chen, Lu and Zhang \cite{clz25,clz25corr} study \lmity\ in cubic graphs. In particular, they establish constant lower bounds on the number of \lm\ vertices in \con2 nonbipartite cubic graphs; we shall mention their results in more detail in due course.
Their work relies on the equivalent definition that a vertex $v$ of a \cc2 $G$ is \lm\ if \mbox{$G-v-N(v)$} is matchable, provided $G$ is not $\Theta$ shown in \cref{fig: theta}; in fact, they consider only simple graphs, and they refer to such a vertex as a \emph{nice} vertex.
In this paper, we exploit our viewpoint stated above in terms of spanning subgraphs and \paritylemma\ to obtain stronger lower bounds that feature matching-theoretic invariants that arise from the work of Lov\'asz \cite{lova87}.

It is worth noting that other researchers have stumbled upon \lmity\ in their investigations pertaining to perfect matchings. For instance, Kr\'al, Sereni and Steibitz \cite{kss09} used them to establish that every \cc2 has at least $\frac{n}{2}$ perfect matchings.

Our next goal is to state a characterization of \lm\ vertices that appears in \cite{clz25}.
To do so, we need some concepts that arise in the study of perfect matchings.
We use $\codd(G)$ to denote the number of odd components (that is, components of odd order) of a graph~$G$.
Let us recall Tutte's \mbox{$1$-factor} Theorem which states that
a graph $G$ is matchable if and only if $\codd(G-S) \leqslant |S|$ for each $S \subseteq V(G)$;
\emph{barriers} are those subsets $S$ that satisfy this inequality with equality.
The following is easily proved using Tutte's Theorem.

\begin{proposition}\label{when is G-u-v matchable}
	For distinct vertices $u$ and $v$ of a matchable graph $G$, the graph $G-u-v$ is matchable if and only if $G$ has no barrier that contains both $u$ and $v$;
	consequently, $G$ is \mc\ if and only if each of its barriers is stable. \qed
\end{proposition}

Using Tutte's Theorem, \paritylemma\ and \cref{when is G-u-v matchable}, one may easily deduce the result of Sch{\"o}nberger which states that every \cc2 is matching covered.
It is also worth noting that if $B$ is a barrier of a \mcg~$G$ then every component of $G-B$ is odd.
We now proceed to state the aforementioned characterization of \lm\ vertices. 
This also appears in \cite[Theorem 2.3]{clz25}, but we provide a proof for the sake of completeness.

\begin{lemma}\label{characterization of non lm vertices}
	A vertex $v$ of a \cc2 $G$ is not \lm\ if and only if there exists a barrier $B$ such that $v$ is an isolated vertex in $G-B$.
\end{lemma}
\begin{proof}
	If $v$ does not have three distinct neighbors, it is easy to see that $v$ is not \lm\ and $N(v)$ is the desired barrier; now, suppose that $v$ has three distinct neighbors.

	First, assume there exists a barrier~$B$ such that $v$ is isolated in $G-B$. We let $B':=B-N(v)$.
	Observe that $\codd(G-v-N(v)-B') > |B'|$. Consequently, $G-v-N(v)$ is not matchable, or equivalently, $v$ is not \lm\ in $G$.

	Conversely, assume that $v$ is not \lm.
	Thus, by Tutte's Theorem and the fact that \mbox{$G-v-N(v)$} is of even order, there exists $S \subseteq V(G-v-N(v))$ such that $\codd(G-v-N(v)-S) \geqslant |S|+2$.
	Observe that $B:=S \cup N(v)$ is the desired barrier in $G$.
\end{proof}

Next, one may ask whether there are nontrivial lower bounds on the number of \lm\ vertices.
For a \cc2 $G$, we let $\Lambda(G)$ and $\lambda(G)$ denote the set and the number of its \lm\ vertices, respectively.
We use $H[A,B]$ to denote a bipartite graph with specified color classes $A$ and $B$, and 
we remark that every connected bipartite cubic graph is in fact \con2.
Clearly, due to the following easy observation, any lower bound on $\lambda(G)$ must evaluate to at most zero in the case of bipartite graphs.

\begin{proposition}\label{lambda is 0 for bipartite graphs}
	In a connected bipartite cubic graph $H[A,B]$, no vertex is \lm, or equivalently, $\lambda(H)=0$. \qed
\end{proposition}

In light of the above, for a connected bipartite cubic graph $H[A,B]$, it makes sense to define a notion of \lmity\ for pairs.
For a pair~$(a,b)$ of vertices, an \emph{\mbox{$(a,b)$-matching}} is any spanning subgraph of~$H$ such that the degrees of $a$ and $b$ are three whereas the degree of every other vertex is one;
furthermore, the pair $(a,b)$ is \emph{\lm} if $H$ admits an \mbox{$(a,b)$-matching}.
Note that an \mbox{$(a,b)$-matching} may exist only if $a$ and $b$ belong to distinct color classes; for the sake of notational convenience, we shall consider ordered pairs $(a,b)$ where $a \in A$ and $b \in B$.
We let $\Rho(H)$ and $\rho(H)$ denote the set and the number of \lm\ pairs in $H$, respectively.
We remark that, in their work, Chen, Lu and Zhang \cite{clz25} defined \lmity\ of pairs.
They work with the equivalent definition that a pair $(a,b)$ is \lm\ if $H-a-b-N(a)-N(b)$ is matchable, provided each of $a$ and $b$ has three distinct neighbors;
in fact, they restrict themselves to simple graphs, and they refer to such a pair as a \emph{nice} pair.
In their work, they establish constant lower bounds on $\rho$; we shall mention them in due course.

In this paper, we shall establish lower bounds on $\rho$ in \cref{sec: lm pairs in bipartite cubic graphs}, and we shall
find some of the intermediate results useful in establishing a lower bound on $\lambda$ in \cref{sec: lm vertices in cubic graphs}.
In both cases, we shall first prove our bounds for the \con3 case, and then extend them to \cc2s using the fact that every \cc2 admits a unique decomposition into \cc3s; see \cref{unique 2cd}.
One would hope to prove a lower bound on $\lambda$ that is linear in the order $n$.
However, such a bound does not exist even for nonbipartite \con3 cubic graphs;
see \cref{no linear bound on lambda}.

In order to prove \cref{no linear bound on lambda}, to explain constructions of various graph families, as well as to prove most of our results, we shall find the operation of `splicing', and its inverse notion `cut-contractions', useful.
Given any cut $\partial(X)$ of a graph $G$, the graph $G/X \rightarrow x$, or simply $G/X$, is obtained from $G$ by shrinking $X$ into a single \emph{contraction vertex} $x$;
the graph $G/\Xbar \rightarrow \xbar$ is defined analogously.
We refer to $G/X$ and $G/\Xbar$ as the \emph{$\partial(X)$-contractions}, or simply \emph{cut-contractions}, of $G$.
For instance, the $\partial(X)$-contractions of the graph shown in \cref{fig: k33 splice k4} are $K_4$~and~$K_{3,3}$.
Observe that $\kappa'(G') \geqslant \kappa'(G)$, where $G'$ is any cut-contraction of $G$;
this, coupled with \cref{for cubic vtx-conn = edge-conn}, yields the following that we shall find useful in our inductive proofs.

\begin{proposition}\label{3-cut contractions preserve connectivity}
	Every \cc2 $G$ satisfies $\kappa(G') \geqslant \kappa(G)$, where $C$ is any $3$-cut and~$G'$ is a $C$-contraction. \qed
\end{proposition}

A cut $\partial(X)$ of a \mcg~$G$ is a \emph{separating cut} if both $\partial(X)$-contractions are also \mc. For instance, it follows from Sch\"onberger's result and \cref{for cubic vtx-conn = edge-conn} that every $3$-cut in a \cc2 is a separating cut.
The following characterization of separating cuts is easily proved; see \cite[Theorem 4.2]{lumu24}.

\begin{proposition}\label{separating cuts characterization}
	A cut $\partial(X)$ of a \mcg\ $G$ is a separating cut if and only if, for each edge $e \in E(G)$, there exists a perfect matching $M$ containing $e$ such that $|\partial(X) \cap M| = 1$. \qed
\end{proposition}

Now, let $G$ and $H$ be disjoint graphs with specified vertices $u$ and $v$, respectively, such that \mbox{$d_G(u)=d_H(v)$}, and
let $\f$ denote any bijection from $\partial_G(u)$ to $\partial_H(v)$.
The graph $J:=(G \odot H)_{u,v,\f}$ is obtained from the union of $G-u$~and~$H-v$ by
adding the following edges: for each edge $e \in \partial_G(u)$, we add the edge~$u'v'$, where
$u'$ is the end of $e$ in $G$ that is distinct from $u$, and
$v'$ is the end of $\f(e)$ in $H$ that is distinct from $v$.
We say that $J$ is obtained by \emph{splicing $G$~at~$u$ with $H$~at~$v$ with respect to~$\f$}, or simply by \emph{splicing} $G$ and $H$, and the cut $\partial_J(V(G)-u)$ is the corresponding \emph{splicing cut}.
Furthermore, if the choices of $u$, $v$ and $\f$ are either irrelevant or clear from the context, we may simplify the notation to $G \odot H$.
For instance, see \cref{fig: k33 splice k4}.

\fig{$K_{3,3} \odot K_4$}{fig: k33 splice k4}{\begin{tikzpicture}[scale=1.625]

\tikzmath{
	\height = 0.6;
	\halfsidelength = sin(30)*\height;
	\sidelength = 2*\halfsidelength;
	\x = 1;
	\y = 0;
}


\foreach \i in {1,2,3}{
	\node[vtx] (b\i) at (\i,1.25) {};
}

\node[vtx] (a2) at (2,0) {};
\node[vtx] (a3) at (3,0) {};

\node[vtx] (v2) at (\x,\y+2*\height/3) {};
\node[vtx] (v1) at (\x-\halfsidelength, \y-\height/3) {};
\node[vtx] (v3) at (\x+\halfsidelength, \y-\height/3) {};


\foreach \i in {1,2,3}
	\foreach \j in {2,3}
		\draw (b\i) -- (a\j);

\draw (v1) -- (v2) -- (v3) -- (v1);

\draw (v1) to[out=105, in=180+53] (b1);
\draw (v2) -- (b2);
\draw plot [smooth,tension=1] coordinates {(v3) ($(a2)!0.5!(b2)$) (b3)};


\draw[thick,magenta] plot [smooth,tension=1.25] coordinates {($(v1)!0.5!(0,0)+(-0.05,0.1)$) ($(v2)!0.35!(b1)$) ($(v3)!0.5!(a2)+(0.05,0.1)$)};
\node[magenta] at ($(v1)!0.5!(0,0)+(-0.42,0.1)$) {$\partial(X)$};
\node at ($(v1)!0.5!(v3)-(0.025,0.22)$) {$X$};

\node[anchor=east] at ($(b1)+(-0.25,0)$) {$B'$};
\draw ($(b1)+(-0.25,0.25)$) rectangle ($(b3)+(0.25,-0.25)$);
\node [right = -0.5mm of a3] {$v$};

\node at (4,0) {};
\end{tikzpicture}}

Note that cut-contractions are unique but splicing is not.
For instance, both the Petersen graph and the pentagonal prism may be obtained by splicing wheels of order six at their hubs.
One may easily verify that the splicing of two \mcg s yields a \mcg, and the corresponding splicing cut is a separating cut.
The following easy observation regarding \lmity\ across separating cuts is a generalization of \cite[Lemma 2.11]{clz25}.

\begin{lemma}\label{lm vertex in a separating cut-contraction is lm in the bigger graph}
	Let $\partial(X)$ denote a separating cut of a \cc2 $G$. Every \lm\ vertex of $G/\Xbar \rightarrow \xbar$, that is not the contraction vertex $\xbar$, is also \lm\ in $G$.
\end{lemma}
\begin{proof}
	Let $v \in X$ such that $G/\Xbar$ has a $v$-matching, say $M_1$.
	By \cref{separating cuts characterization}, there exists a perfect matching~$M_2$ of~$G/X$ that contains the unique edge in $M_1 \cap \partial_{G}(X)$.
	Observe that $M_1 \cup M_2$ is the desired $v$-matching in $G$.
\end{proof}

It is easy to see that every vertex is \lm\ in $K_4$, and by the above lemma, every vertex in~$X$ is \lm\ in the graph shown in \cref{fig: k33 splice k4}. The reader may verify that every vertex in $B'$ is \lm\ as well. Note that $B'$ is a barrier; therefore, by \cref{characterization of non lm vertices}, it is clear that $\Lambda(X)=B' \cup X$. We shall find this fact useful in proving the negative result. We shall also find the following result useful throughout this paper, including the proof of the negative result.

\begin{proposition}\label{splicing of two 3-conn cubic graphs is 3-conn}
	Any splicing of \con3 cubic graphs is also (cubic and) \con3.
\end{proposition}
\begin{proof}
	Let $G:=G_1 \odot G_2$, where $G_1$ and $G_2$ are \cc3s, and
	let $\partial_G(X)$ be the corresponding splicing cut.
	Adjust notation so that $G_1=G/\Xbar \rightarrow \xbar$ and $G_2=G/X \rightarrow x$.
	Clearly, if either of $G_1$ and $G_2$ is $\Theta$, the desired conclusion holds trivially.
	Now, suppose that neither of them is~$\Theta$; whence $n(G) \geqslant 4$ and $\partial_G(X)$ is a matching.
	We shall argue that, for each pair of distinct \mbox{vertices $u,v \in V(G)$}, the graph $G-u-v$ is connected;
	adjust notation so that $u \in X$.

	First, suppose that $v \in X$.
	Since $G_1$ is \con3, $(G-u-v)/\Xbar=G_1-u-v$ is connected; consequently, since $G[\Xbar]=G_2-x$ is connected, so is $G-u-v$.
	Now, suppose that $v \in \Xbar$.
	Since $G_1$ and $G_2$ are \con3, $G[X-u]=G_1-u-\xbar$ and $G[\Xbar-v]=G_2-v-x$ are connected.
	Since $\partial_G(X)$ is a matching, at least one of its edges exists in $G-u-v$. Using these facts, we infer that $G-u-v$ is connected.
\end{proof}

We are now ready to state and prove the aforementioned negative result which implies that there is no lower bound on the number of \lm\ vertices that is linear in the order $n$.

\begin{proposition}\label{no linear bound on lambda}
	For every even integer $n \geqslant 12$, there exists a \con3 nonbipartite cubic graph~$G$ of order $n$ such that $\lambda(G)=6$.
\end{proposition}
\begin{proof}
	Let $G$ be any graph obtained by splicing $K_{3,3} \odot K_4$, shown in \cref{fig: k33 splice k4}, at $v$ with any \bcc3 $H[A,B]$ of order $n-6$ at a vertex $b \in B$.
	An abstract drawing of such a graph is shown in \cref{fig: negative result}.
	By \cref{splicing of two 3-conn cubic graphs is 3-conn}, $G$ is a \cc3.
	Note that $B' \cup B - b$ and $A$ are barriers of $G$. Clearly by \cref{characterization of non lm vertices}, each vertex except for the ones in $B' \cup X$ are not \lm\ in~$G$.
	By \cref{lm vertex in a separating cut-contraction is lm in the bigger graph}, it is clear that $\Lambda(G)=B' \cup X$; whence, $\lambda(G)=6$.
\end{proof}

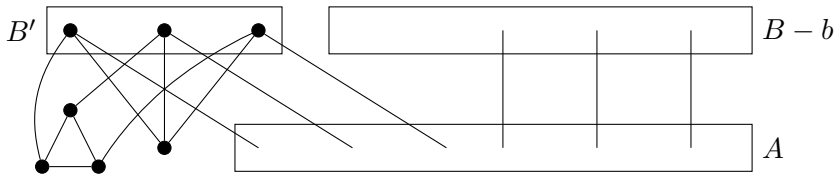
\fig{An illustration for the proof of \cref{no linear bound on lambda}}{fig: negative result}{\begin{tikzpicture}[scale=1.6]

\tikzmath{
	\height = 0.6;
	\halfsidelength = sin(30)*\height;
	\sidelength = 2*\halfsidelength;
	\x = 1;
	\y = 0;
}


\foreach \i in {1,2,3}{
	\node[vtx] (b\i) at (\i,1.25) {};
}

\node[vtx] (a2) at (2,0) {};
\coordinate (a3) at (3,0);

\node[vtx] (v2) at (\x,\y+2*\height/3) {};
\node[vtx] (v1) at (\x-\halfsidelength, \y-\height/3) {};
\node[vtx] (v3) at (\x+\halfsidelength, \y-\height/3) {};


\draw ($(b1)+(2.75,0.25)$) rectangle ($(b1)+(7.25,-0.25)$);
\draw ($(a1)+(1.75,0.25)$) rectangle ($(a1)+(7.25,-0.25)$);


\foreach \i in {1,2,3}
		\draw (b\i) -- (a2);

\draw (v1) -- (v2) -- (v3) -- (v1);

\draw (v1) to[out=105, in=180+53] (b1);
\draw (v2) -- (b2);
\draw plot [smooth,tension=1] coordinates {(v3) ($(a2)!0.5!(b2)$) (b3)};

\draw (b1) -- (a3);
\draw (b2) -- (4,0);
\draw (b3) -- (5,0);

\draw (5.6,0) -- (5.6,1.25);
\draw (6.6,0) -- (6.6,1.25);
\draw (7.6,0) -- (7.6,1.25);

\node[anchor=east] at ($(b1)+(-0.25,0)$) {$B'$};
\draw ($(b1)+(-0.25,0.25)$) rectangle ($(b3)+(0.25,-0.25)$);
\node[anchor=west] at ($(b1)+(7.25,0)$) {$B-b$};
\node[anchor=west] at ($(a1)+(7.25,0)$) {$A$};

\node at (-0.25,0) {};
\end{tikzpicture}}

Consequently, our bound for $\lambda$ features a matching-theoretic invariant that arises from the work of Lov\'asz~\cite{lova87}.
A cut $C$ of a \mcg\ $G$ is \emph{tight} if $|C \cap M| = 1$ for each perfect matching~$M$.
It is easy to see that every tight cut is a separating cut. 
The converse, however, holds only for bipartite graphs; for instance, the triangular prism $\overline{C_6}$ has a separating $3$-cut that is not tight.
A cut is \emph{trivial} if either shore contains at most one vertex; otherwise, it is \emph{nontrivial}.
Ergo, if $C$ is a nontrivial tight cut, both its $C$-contractions are smaller \mcg s.
If either $C$-contraction has a nontrivial tight cut, one may repeat this process recursively.
Clearly, at the end, one obtains a list of matching covered graphs, each of which is free of nontrivial tight cuts.
This procedure is called the \emph{tight cut decomposition procedure}.
For instance, an application of the tight cut decomposition procedure to $K_{3,3} \odot K_4$, shown in \cref{fig: k33 splice k4}, yields one $K_{3,3}$ and one $K_4$.

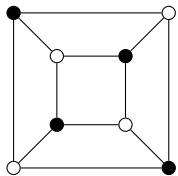
\fig{The \emph{cube} graph}{fig: cube}{\begin{tikzpicture}[scale=1.75]

\tikzmath{
	\one = 0.365;
	\two = 0.825;
}


\node[vtx] (a1) at (-\two,  \two) {};
\node[vtx-white] (a2) at ( \two,  \two) {};
\node[vtx] (a3) at ( \two, -\two) {};
\node[vtx-white] (a4) at (-\two, -\two) {};

\node[vtx-white] (b1) at (-\one,  \one) {};
\node[vtx] (b2) at ( \one,  \one) {};
\node[vtx-white] (b3) at ( \one, -\one) {};
\node[vtx] (b4) at (-\one, -\one) {};


\draw (a1) -- (a2) -- (a3) -- (a4) -- (a1);
\draw (b1) -- (b2) -- (b3) -- (b4) -- (b1);
\foreach \i in {1,...,4}
	\draw (b\i) -- (a\i);

\end{tikzpicture}}

A \mcg\ free of nontrivial tight cuts is a \emph{brace} if it is bipartite; otherwise it is a \emph{brick}.
The smallest cubic braces are $\Theta, \cfour, K_{3,3}$ and the \emph{cube} graph, whereas the smallest cubic bricks are $K_4$ and~$\overline{C_6}$; see \cref{fig: smallest 2-conn cubic graphs,fig: cube}.
It is easy to see that all cubic bricks, as well as cubic braces of order six or more, are simple.
Lov\'asz \cite{lova87} proved the following remarkable result.

\begin{theorem}\label{tcd uniqueness}
	Any two applications of the tight cut decomposition procedure to a \mcg\ $G$ yield the same list of bricks and braces (up to multiplicities of edges).
\end{theorem}

By \emph{bricks and braces of} a \mcg\ $G$, denoted as $\bstar(G)$, we mean the multiset that comprises the underlying simple graphs of the bricks and braces obtained from~$G$ by any application of the tight cut decomposition procedure.
Furthermore, we use $b(G)$ to denote the number of its bricks, and for reasons that will become clear soon, we use $b'(G)$ to denote the number of its braces of order six or more.
We remark that Kothari and Murty \cite{km16} employed \lmity\ to prove that a matching covered graph is \mbox{`$J$-free'} if and only if each of its bricks is \mbox{`$J$-free'}, where $J$ is any cubic brick; we refer the interested reader to their paper for more details.
The following is easy to prove using \cref{characterization of tight cuts in bmcgs} that appears in \cref{sec: lm pairs in bipartite cubic graphs}.

\begin{proposition}\label{mcg is bipartite iff b=0}
	A matching covered graph $G$ is bipartite if and only if $b(G)=0$. \qed
\end{proposition}

Interestingly, for \cc2s, the following is well-known; see \cite[Theorem 5.8]{lumu24}.

\begin{proposition}\label{every tight cut of a cc2 is a 3-cut}
	Every tight cut of a \cc2 is a $3$-cut.
\end{proposition}

Consequently, we have the following strengthening of \cref{tcd uniqueness} which states that the multiplicities of edges is immaterial in the case of cubic graphs.

\begin{corollary}
	Any two applications of the tight cut decomposition procedure to a \cc2~$G$ yield the same list of cubic bricks and cubic braces. \qed
\end{corollary}

Our lower bounds for $\lambda$ as well as $\rho$ feature derivatives of $\bstar(G)$.
In order to motivate them, we shall find alternative characterizations of bricks and braces, and their implications, useful.
A graph $G$, of order four or more, is \emph{bicritical} if $G-u-v$ is matchable for each pair of distinct vertices $u$~and~$v$. Edmonds, Lov\'asz and Pulleyblank \cite{elp82} established the following characterization of bricks.

\begin{theorem}\label{brick iff bicritical and 3-conn}
	A graph is a brick if and only if it is \con3 and bicritical.
\end{theorem}

It follows from \cref{when is G-u-v matchable} that a matchable graph, of order at least four, is bicritical if and only if each of its barriers is \emph{trivial} --- that is, either singleton or empty. This observation, along with the above theorem and \cref{for cubic vtx-conn = edge-conn}, implies the following.

\begin{corollary}\label{a 2 conn cubic graph is not theta or brick iff there is a nontrivial barrier}
	A \con2 cubic graph has a nontrivial barrier if and only if it is neither $\Theta$ nor a brick. \qed
\end{corollary}

Our first matching-theoretic invariant $\B(G)$, defined as the sum of the orders of the bricks of a matching covered graph $G$, is inspired by the following easy observation.

\begin{proposition}\label{in a cubic brick every vertex is lm}
	Every vertex of a cubic brick $G$ is \lm, or equivalently, \mbox{$\lambda(G)=\B(G)=n(G)$}.
\end{proposition}
\begin{proof}
	Since $G$ is bicritical, any vertex $v$ has three distinct neighbors $v_1,v_2$ and $v_3$, and $G-v_1-v_2$ has a perfect matching, say $M$. Note that $M + vv_1 + vv_2$ is the desired $v$-matching.
\end{proof}

We are now ready to state one of our main results that establishes a lower bound on $\lambda$ for all \cc3s.
It implies a result of \cite{clz25} which states that every \con3 nonbipartite cubic graph, distinct from $K_4$, satisfies $\lambda \geqslant 6$.

\begin{theorem}\label{lambda lower bound 3-conn - intro}
	Every \cc3\ $G$ satisfies the lower bound $\lambda(G) \geqslant \B(G)$. Furthermore, if $G$ is bipartite then $\lambda(G)=\B(G)=0$; on the other hand, if $G$ is nonbipartite, $\lambda(G) = \B(G)$ if and only if $\B(G) = n(G)$.
\end{theorem}

A proof of the above appears in \cref{sec: lm vertices in cubic graphs} wherein we also provide a structural characterization, using splicing, of those graphs that attain the lower bound with equality.
We remark that the same bound applies to all \cc2s; however, the characterization of tight examples requires an additional invariant that we are yet to define.
For instance, it is easy to verify that $\Lambda=X$ for the graph shown in \cref{fig: k4 glue k33}; therefore, it satisfies $\lambda=\beta=4$, but clearly, $n > \lambda$.

\fig{A \cc2 with $4 = \lambda = \beta < n = 10$}{fig: k4 glue k33}{\begin{tikzpicture}[scale=1.75,xscale=-1]

\foreach \i in {0,1,2} {
	\node[vtx] (a\i) at (\i,1) {};
	\node[vtx] (b\i) at (\i,0) {};
}
\tikzmath{
	\x = 0.5 / tan(30);
}
\node[vtx] (v1) at (3,1) {};
\node[vtx] (v2) at (3,0) {};
\node[vtx] (v3) at (3+\x,0.5) {};
\node[vtx] (v4) at (3+0.33*\x,0.5) {};

\draw (v1) -- (v4);
\draw (v2) -- (v4);
\draw (v3) -- (v1);
\draw (v3) -- (v2);
\draw (v3) -- (v4);

\foreach \i in {0,1,2} {
	\foreach \j in {0,1,2} {
		\ifnum \i=2
			\ifnum \j=2
				\draw (a\i) -- (v1);
				\draw (b\j) -- (v2);
			\else
				\draw (a\i) -- (b\j);
			\fi
		\else
			\draw (a\i) -- (b\j);
		\fi
	}
}

\draw[thick, magenta] ($(v1)!0.5!(a2)+(0,0.25)$) -- ($(v2)!0.5!(b2)+(0,-0.25)$);
\node[magenta] at ($(v1)!0.5!(a2)+(0.075,0.425)$) {$\partial(X)$};
\node at (3+0.7*\x,1) {$X$};

\end{tikzpicture}}

In $K_{3,3} \odot K_4$, shown in \cref{fig: k33 splice k4}, $n = 8 > \lambda(G) = 6 > \B(G) = 4$.
\cref{fig: k33 spliced with 3 k4s} shows a \cc3 that is not a brick but satisfies $\lambda = \B = n$. The reader may use the tight cuts $C_1,C_2$ and $C_3$ to compute the value of $\B$.

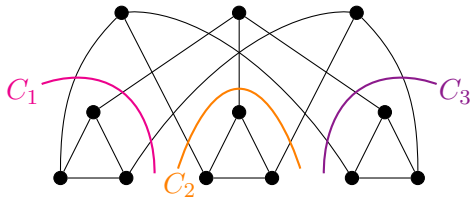
\fig{The smallest \cc3 that is not a brick but satisfies $\lambda=\B=n$}{fig: k33 spliced with 3 k4s}{\begin{tikzpicture}[scale=1.6]

\tikzmath{
	\height = 0.7;
	\halfsidelength = sin(30)*\height;
	\sidelength = 2*\halfsidelength;
	\offset = 0.3;
	\y = -0.7;
}


\foreach \i in {1,2,3}{
	\node[vtx] (b\i) at (1.25*\i,0.825) {};
	\tikzmath{ \x = 1.25*\i+\offset*(\i-2); }
	\node[vtx] (a\i2) at (\x,\y+2*\height/3) {};
	\node[vtx] (a\i1) at (\x-\halfsidelength, \y-\height/3) {};
	\node[vtx] (a\i3) at (\x+\halfsidelength, \y-\height/3) {};
}


\foreach \i in {1,2,3}{
	\draw (b2) -- (a\i2);
	\draw (a\i1) -- (a\i2) -- (a\i3) -- (a\i1);
}
\draw (a21) -- (b1);
\draw (a23) -- (b3);

\draw (a11) to[out=90+0, in=180+45] (b1);
\draw (a33) to[out=90-0, in=360-45] (b3);

\draw plot [smooth,tension=1.1] coordinates {(a13) ($(a22)!0.6!(b2)$) (b3)};
\draw plot [smooth,tension=1.1] coordinates {(a31) ($(a22)!0.6!(b2)$) (b1)};


\draw[thick, magenta] plot [smooth,tension=1.1] coordinates {
	($(a11)+(-0.2,1)$)
	($(a12)!0.2!(b2)$)
	($(a13)+(0.3,0.05)$)
};
\node[magenta] at ($(a11)+(-0.4,0.92)$) {$C_1$};
\draw[thick, violet!90] plot [smooth,tension=1.1] coordinates {
	($(a33)+(0.2,1)$)
	($(a32)!0.2!(b2)$)
	($(a31)+(-0.3,0.05)$)
};
\node[violet!90] at ($(a33)+(0.4,0.92)$) {$C_3$};
\draw[thick, orange] plot [smooth, tension=1.1] coordinates {
	($(a21)+(-0.3,0.1)$)
	($(a22)+(0,0.25)$)
	($(a23)+(0.3,0.1)$)
};
\node[orange] at ($(a21)+(-0.28,-0.075)$) {$C_2$};

\end{tikzpicture}}

It is worth noting that there are \cc3s that satisfy $\lambda=n$ and $\lambda > \beta$.
For instance, consider the graph obtained from $K_{3,3} \odot K_{3,3}$, shown in \cref{fig: k33 splice k33}, by splicing with $K_4$'s at each vertex of $A' \cup B'$.
For the graph thus obtained, shown in \cref{fig: lambda > beta example}, $\lambda=n=18$ whereas $\B=16$.
This is particularly interesting in light of the last part of \cref{lambda lower bound 3-conn - intro}.
This inspired us to characterize all \cc2s that satisfy $\lambda=n$, and we do this in \cref{sec: lm vertices in cubic graphs}, thus resolving a problem posed by Chen, Lu and Zhang \cite[Problem 2]{clz25}.

\fig{The smallest \con3 cubic graph with $n = \lambda > \beta$}{fig: lambda > beta example}{\begin{tikzpicture}[scale=1.75]
\begin{scope}[rotate=90]
	\tikzmath{
			\sidelength = 0.7;
			\halfsidelength = \sidelength/2;
			\height = sin(60)*\sidelength;
			\offset = 0.65;
			\xcoord = 1;
			\ycoord = 0.4;
		}
	\coordinate (b11) at (-\offset + \xcoord-\halfsidelength,\ycoord-\height);
	\coordinate (b12) at (-\offset + \xcoord,\ycoord);
	\coordinate (b13) at (-\offset + \xcoord+\halfsidelength,\ycoord-\height);

	\coordinate (b21) at (\offset + \xcoord-\halfsidelength,\ycoord-\height);
	\coordinate (b22) at (\offset + \xcoord,\ycoord);
	\coordinate (b23) at (\offset + \xcoord+\halfsidelength,\ycoord-\height);

	\coordinate (a1) at (1 - 1, 1);
	\coordinate (a2) at (1,1);
	\coordinate (a3) at (1 + 1, 1);

	\draw (b11) -- (b12) -- (b13) -- (b11);
	\draw (b21) -- (b22) -- (b23) -- (b21);
	\draw (b11) -- (a1);
	\draw (b12) -- (a2);
	\draw (b13) -- (a3);
	\draw (b21) -- (a1);
	\draw (b22) -- (a2);
	\draw (b23) -- (a3);

	\begin{scope}[yshift=74, yscale=-1]
		\tikzmath{
			\sidelength = 0.7;
			\halfsidelength = \sidelength/2;
			\height = sin(60)*\sidelength;
			\offset = 0.65;
			\xcoord = 1;
			\ycoord = 0.4;
		}
		\coordinate (2b11) at (-\offset + \xcoord-\halfsidelength,\ycoord-\height);
		\coordinate (2b12) at (-\offset + \xcoord,\ycoord);
		\coordinate (2b13) at (-\offset + \xcoord+\halfsidelength,\ycoord-\height);

		\coordinate (2b21) at (\offset + \xcoord-\halfsidelength,\ycoord-\height);
		\coordinate (2b22) at (\offset + \xcoord,\ycoord);
		\coordinate (2b23) at (\offset + \xcoord+\halfsidelength,\ycoord-\height);

		\coordinate (2a1) at (1 - 1, 1);
		\coordinate (2a2) at (1,1);
		\coordinate (2a3) at (1 + 1, 1);

		\draw (2b11) -- (2b12) -- (2b13) -- (2b11);
		\draw (2b21) -- (2b22) -- (2b23) -- (2b21);
		\draw (2b11) -- (2a1);
		\draw (2b12) -- (2a2);
		\draw (2b13) -- (2a3);
		\draw (2b21) -- (2a1);
		\draw (2b22) -- (2a2);
		\draw (2b23) -- (2a3);
	\end{scope}

	\draw (a1) -- (2a1);
	\draw (a2) -- (2a2);
	\draw (a3) -- (2a3);

	\draw (a1)node[vtx]{} (a2)node[vtx]{} (a3)node[vtx]{} (b11)node[vtx]{} (b12)node[vtx]{} (b13)node[vtx]{} (b21)node[vtx]{} (b22)node[vtx]{} (b23)node[vtx]{};

	\draw (2a1)node[vtx]{} (2a2)node[vtx]{} (2a3)node[vtx]{} (2b11)node[vtx]{} (2b12)node[vtx]{} (2b13)node[vtx]{} (2b21)node[vtx]{} (2b22)node[vtx]{} (2b23)node[vtx]{};
\end{scope}
\end{tikzpicture}}

We now switch our attention to bipartite graphs.
We begin by stating characterizations of \bmcg s that are well-known and easily proved using Hall's Theorem; see \cite[Theorems 2.9 and 3.8]{lumu24}.

\begin{proposition}\label{characterization of bmcgs}
	For a bipartite matchable graph $H[A,B]$, the following are equivalent:
	\begin{enumerate}[label=(\roman*)]
		\item $H$ is matching covered,
		\item $H-a-b$ is matchable for every pair $a \in A$ and $b \in B$, and
		\item $|N_H(A')| \geqslant |A'|+1$ for every nonempty $A' \subset A$. \qed
	\end{enumerate}
\end{proposition}

Recall that a brace is a \bmcg\ that is free of nontrivial tight cuts. They admit similar sounding characterizations that may be easily deduced from \cref{characterization of tight cuts in bmcgs}; see \cite[Theorems~5.16~and~5.17]{lumu24}.

\begin{proposition}\label{characterization of braces}
	For a \bmcg\ $H[A,B]$, the following are equivalent:
	\begin{enumerate}[label=(\roman*)]
		\item $H$ is a brace,
		\item $H-a_1-a_2-b_1-b_2$ is matchable for distinct $a_1,a_2 \in A$ and $b_1,b_2 \in B$, and
		\item $|N_H(A')| \geqslant |A'|+2$ for every nonempty subset $A'$ of $A$ such that $|A'| \leqslant |A|-2$. \qed
	\end{enumerate}
\end{proposition}

In the case of connected bipartite cubic graphs, one might hope to prove a lower bound on the number of \lm\ pairs that is quadratic in the order~$n$.
However, as we shall see later, such a bound does not exist even for \bcc3s; see \cref{rho lower bound 3-conn - easy direction}. 
Inspired by \cref{in a simple brace every pair is lm} stated below, we define the invariant $\b$ of a matching covered graph $G$ as follows:
\[ \b(G) := \sum \left(\frac{n(H)}{2}\right)^2 \]
wherein the sum is over all braces $H$ (of $G$) of order six or more.

\begin{proposition}\label{in a simple brace every pair is lm}
	Every pair $(a,b)$ of a cubic brace $H[A,B]$ distinct from $\cfour$, where $a \in A$ and $b \in B$, is \lm, or equivalently, if $H \neq \Theta$ then $\rho(H)=\b(H)$.
\end{proposition}
\begin{proof}
	Observe that $\rho(\Theta)=1$ whereas $\b(\Theta)=0$. Now suppose that $H$ is a cubic brace of order six or more; in particular, it is simple.
	By \cref{characterization of braces} $(ii)$, the graph $H-a_1-a_2-b_1-b_2$ has a perfect matching~$M$ for distinct $a_1,a_2 \in N_H(b)-a$ and $b_1,b_2 \in N_H(a)-b$.
	Observe that $M+ab_1+ab_2+ba_1+ba_2$ is an $(a,b)$-matching in $H$.
\end{proof}

We are now ready to state our lower bounds for the number of \lm\ pairs in \bcc3s;
this, along with \cref{lm pairs across a 2-cut}, implies \cite[Corollary 4.2]{clz25} which states that every simple connected bipartite cubic graph satisfies $\rho \geqslant 9$ and equality holds if and only if it is $K_{3,3}$.

\begin{theorem}\label{rho lower bound 3-conn - intro}
	Every \bcc3\ $H$ satisfies:
	\[ \rho(H) ~\geqslant~ \b(H) + 3b'(H) - 3 ~\geqslant~ 3n(H) -9 \]
\end{theorem}

A proof of the above appears in \cref{sec: lm pairs in bipartite cubic graphs} wherein we also provide structural characterizations, using splicing, of those graphs that attain either lower bound with equality. \cref{fig: example for lower bound on rho} shows a few examples; in each of these graphs, the only tight cut $C$ is shown, and the set $A' \times B'$ comprises the only non \lm\ pairs.

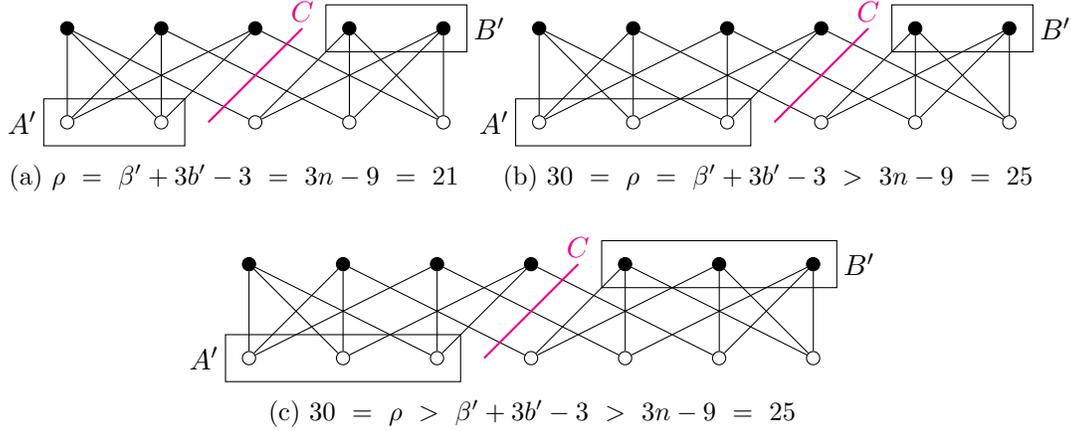
\fig{The graphs $K_{3,3} \odot K_{3,3}$, \emph{cube} $\odot\;K_{3,3}$ and \emph{cube} $\odot$ \emph{cube}}{fig: example for lower bound on rho}{
\subfig
	{0.45,0.55}
	{$\rho ~=~ \b + 3b' - 3 ~=~ 3n - 9 ~=~ 21$ , $30 ~=~ \rho ~>~ \b + 3b' - 3 ~=~ 28 ~>~ 3n - 9 ~=~ 27$}
	{fig: k33 splice k33,fig: cube splice k33}
	{k33-splice-k33.tikz,cube-splice-k33.tikz}
\\ \bigskip 
\subfig
	{1}
	{$40 ~=~ \rho ~>~ \b + 3b' - 3 ~=~ 35 ~>~ 3n - 9 ~=~ 33$}
	{fig: cube splice cube}
	{cube-splice-cube.tikz}
}

It is worth noting that the aforementioned bounds on $\rho$ do not hold for all connected bipartite cubic graphs.
For instance, for the graph shown in \cref{fig: k33 glue theta glue k33}, $P= \{(a,b)\} \cup (A_1 \times B_1) \cup (A_2 \times B_2)$; whence, $\rho=19$ whereas, $\b+3b'-3=21$ and $3n-9=33$.

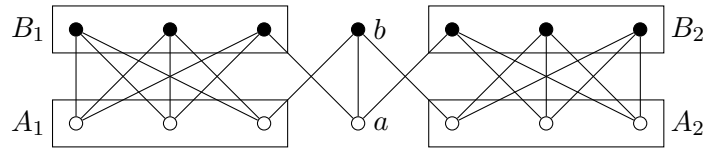
\fig{A connected bipartite cubic graph whose \con3 pieces are $\Theta$ and two copies of $K_{3,3}$}{fig: k33 glue theta glue k33}{\begin{tikzpicture}[scale=1.5]

	\foreach \i in {0,...,6}{
		\node[vtx-white] (a\i) at (\i,0) {};
		\node[vtx] (b\i) at (\i,1) {};
	}

	\foreach \i in {0,1,2}{
		\foreach \j in {0,1,2}{
			\ifnum \i=2
				\ifnum \j=2
					\draw (a\i) -- (b3);
					\draw (b\j) -- (a3);
				\else
					\draw (a\i) -- (b\j);
				\fi
			\else
				\draw (a\i) -- (b\j);
			\fi
			\pgfmathtruncatemacro{\i}{\i + 4}		
			\pgfmathtruncatemacro{\j}{\j + 4}
			\ifnum \i=4
				\ifnum \j=4
					\draw (a\i) -- (b3);
					\draw (b\j) -- (a3);
				\else
					\draw (a\i) -- (b\j);
				\fi
			\else
				\draw (a\i) -- (b\j);
			\fi
		}
	}

	\draw (a3) -- (b3);

	\draw ($(a0)+(-0.25,0.25)$) rectangle ($(a2)+(0.25,-0.25)$);
	\node [left = 2mm of a0] {$A_1$};
	\draw ($(b0)+(-0.25,0.25)$) rectangle ($(b2)+(0.25,-0.25)$);
	\node [left = 2mm of b0] {$B_1$};
	\draw ($(a4)+(-0.25,0.25)$) rectangle ($(a6)+(0.25,-0.25)$);
	\node [right = 2mm of a6] {$A_2$};
	\draw ($(b4)+(-0.25,0.25)$) rectangle ($(b6)+(0.25,-0.25)$);
	\node [right = 2mm of b6] {$B_2$};

	\node [right = -0.25mm of a3] {$a$};
	\node [right = -0.25mm of b3] {$b$};

\end{tikzpicture}}

In order to describe our extensions of \cref{rho lower bound 3-conn - intro,lambda lower bound 3-conn - intro} to all \cc2s, we first discuss a decomposition of such graphs into \cc3s.
For a graph of even order, a cut is \emph{even} if both of its shores have even cardinality; otherwise, it is \emph{odd}.
Let $G$ be any \cc2\ that is not \con3.
By \cref{for cubic vtx-conn = edge-conn}, $G$ contains a \mbox{$2$-cut}, say $C$;
by \paritylemma, $C$ is an even cut.
Clearly, $G-C$ has precisely two components, and from either of them, one may construct a smaller cubic graph by adding an edge joining the vertices of degree two.
Inspired by the terminology in \cite[Section 9.4]{bomu08}, we refer to these smaller (cubic) graphs as the \emph{marked \mbox{\mbox{$C$-components}}} of~$G$;
if either of them is not \con3, then one may repeat this process recursively.
Clearly, at the end of this procedure, that we refer to as the \emph{\mbox{$2$-cut} decomposition procedure}, one obtains a list of \cc3s.
Depending on the choice of \mbox{$2$-cuts} made at each step, $G$ may admit multiple applications of this procedure.
Interestingly, the following is well-known;
it follows from \cite[Theorem 1]{cued80} as well as \cite[Theorem 1.16]{nmgk25}.

\begin{theorem}\label{unique 2cd}
	Any two applications of the \mbox{$2$-cut} decomposition procedure to a \cc2\ $G$ yield the same list of \cc3s.
\end{theorem}

For a \cc2 $G$, we refer to these \cc3s as its \emph{\con3 pieces}.
We use $n_{\sf nonbip}(G)$ to refer to the sum of the orders of the nonbipartite \con3 pieces of~$G$; this invariant shall appear in our aforementioned extension of \cref{lambda lower bound 3-conn - intro}.
For instance, the \con3 pieces of the graph shown in \cref{fig: k4 glue k33} are one $K_4$ and one $K_{3,3}$; consequently, $n_{\sf nonbip}=4$.
Using the fact that, for a \cc2, $|M \cap C| \in \{0,2\}$, where $M$ is any perfect matching and $C$ is any $2$-cut, the following is easily proved.

\begin{proposition}\label{cc2 is bipartite iff each marked C-component is bipartite}
	A \cc2 $H$ is bipartite if and only if both of its marked \mbox{$C$-components} are bipartite, where $C$ is any $2$-cut; consequently, $H$ is bipartite if and only if \mbox{$n_{\sf nonbip}(H)=0$}. \qed
\end{proposition}

We state one more result that relates the bricks and braces of a \cc2 to those of its marked \mbox{$C$-components} with respect to a given $2$-cut $C$; we shall find it useful in proving our extensions of \cref{lambda lower bound 3-conn - intro,rho lower bound 3-conn - intro}.

\begin{proposition}\label{bstar across even 2-cut}
	For a \cc2 $G$ that has a $2$-cut $C$, precisely one of the following holds, where $G_1$ and $G_2$ are its marked \mbox{$C$-components}:
	\begin{enumerate}[label=(\roman*)]
		\item neither $G_1$ nor $G_2$ is $\Theta$, and $\bstar(G) \,=\, \bstar(G_1)~\cupdot~\bstar(G_2)~\cupdot~\{C_4\}$, or
		\item precisely one of $G_1$ and $G_2$, say $G_1$, is $\Theta$, and $\bstar(G) \,=\, \bstar(G_2)~\cupdot~\{C_4\}$, or
		\item both $G_1$ and $G_2$ are $\Theta$, and $G$ is $\cfour$ and $\bstar(G) = \{C_4\}$.
	\end{enumerate}
	Consequently, \hspace{1mm} $b(G)=b(G_1)+b(G_2)$, \hspace{1mm} $\B(G)=\B(G_1)+\B(G_2)$, $b'(G)=b'(G_1)+b'(G_2)$, and \mbox{$\b(G)=\b(G_1)+\b(G_2)$}.
\end{proposition}
\begin{proof}
	Let $X:=V(G_1)$ and let $uv \in C$ so that $u \in X$. Let $Y_1:=X-\{u\}$ and $Y_2:=\Xbar - \{v\}$. Observe that $\partial_G(Y_1)$ and $\partial_G(Y_2)$ are tight cuts of $G$ such that $Y_1 \cap Y_2 = \emptyset$. Observe that $G/\overline{Y_1}=G_1$, $G/\overline{Y_2}=G_2$ and $(G/Y_1)/Y_2$ is $\cfour$. For each $i \in \{1,2\}$, the tight cut $\partial_G(Y_i)$ is nontrivial if and only if $G_i$ is not $\Theta$. With these facts, the reader may verify that the desired conclusion holds in each case.
\end{proof}

We remark that the above result may be generalized to all \mcg s using even $2$-cuts but we choose to omit it.
We now state our extension of \cref{lambda lower bound 3-conn - intro} that applies to all \cc2s.
It implies \cite[Theorem 3.1 $(i)$]{clz25} which states that every \con2 nonbipartite cubic graph satisfies $\lambda \geqslant 4$, and provides a complete characterization of the tight examples.

\begin{theorem}\label{lambda lower bound 2-conn - intro}
	Every \cc2~$G$ satisfies the lower bound $\lambda(G) \geqslant \B(G)$ and equality holds if and only if $\B(G) = n_{\sf nonbip}(G)$.
\end{theorem}

A proof of the above appears in \cref{sec: lm vertices in cubic graphs}.
In order to state our extension of \cref{rho lower bound 3-conn - intro}, we need two other invariants that also arise from \cref{unique 2cd}.
For a \cc2 $G$, we use $\theta(G)$ to denote the number of its \con3 pieces that are isomorphic to~$\Theta$, and $\thetabar(G)$ to denote the number of them that are not isomorphic to $\Theta$. For instance, $\theta(H)=1$ and $\thetabar(H)=2$ for the graph shown in \cref{fig: k33 glue theta glue k33}.

The following is a consequence of \cref{unique 2cd} that we shall find useful later.

\begin{corollary}\label{theta thetabar and n-nonbip over a 2 cut}
	For a \cc2\ $G$ that has a $2$-cut $C$, the following hold, where $G_1$ and $G_2$ are its marked \mbox{$C$-components}: $(i)$ $\theta(G)=\theta(G_1)+\theta(G_2)$,
	$(ii)$ $\thetabar(G)=\thetabar(G_1)+\thetabar(G_2)$, and
	\mbox{$(iii)$ $n_{\sf nonbip}(G)=n_{\sf nonbip}(G_1)+n_{\sf nonbip}(G_2)$} \qed
\end{corollary}

The following is our extension of \cref{rho lower bound 3-conn - intro}.

\begin{theorem}\label{rho lower bound 2-conn - intro}
	Every connected bipartite cubic graph $H$ satisfies:
	\[ \rho(H) ~\geqslant~ \b(H) + 3b'(H) - 3\thetabar(H)+\theta(H) ~\geqslant~ 3n(H) -9\thetabar(H)-5\theta(H)\]
\end{theorem}

For instance, for the graph shown in \cref{fig: k33 glue theta glue k33}, $\rho ~=~ \b + 3b' - 3\thetabar +\theta$ $ ~=~ 3n - 9\thetabar -5\theta ~=~ 19$.
A proof of the above appears in \cref{sec: lm pairs in bipartite cubic graphs}.
For both of the above theorems, we use $2$-cuts as an induction tool, and shall also provide structural characterizations of the tight examples in terms of their \con3 pieces, in their corresponding sections.

\section{\lm\ pairs in bipartite cubic graphs}\label{sec: lm pairs in bipartite cubic graphs}
In order to establish the lower bound stated in \cref{rho lower bound 3-conn - intro}, we shall use tight cuts as our induction tool.
To this end, we first discuss a characterization of tight cuts, or equivalently, of separating cuts in bipartite graphs.
Since any tight cut is an odd cut, let us consider an odd cut $\partial(X)$ of a \bmcg\ $H[A,B]$.
Clearly, one of $X \cap A$ and $X \cap B$ is larger than the other; we use $X^+$~and~$X^-$ to denote the larger and the smaller sets, respectively, and
$\Xbar^+$~and~$\Xbar^-$ are defined analogously.
The following is a well-known characterization of tight cuts that is easily proved, and implies \cref{mcg is bipartite iff b=0}; see \cite[Theorem 4.8]{lumu24}.

\begin{proposition}\label{characterization of tight cuts in bmcgs}
	An odd cut $\partial(X)$ of a \bmcg\ $H[A,B]$ is tight if and only if: $(i)$~$|X^+| = |X^-|+1$, and $(ii)$ there is no edge that has one end in $X^-$ and its other end in $\Xbar^-$. Consequently, if $\partial(X)$ is a tight cut, then both  \mbox{$\partial(X)$-contractions} are bipartite. \qed
\end{proposition}

Using the above, and an easy counting argument, one may easily prove that in a connected bipartite cubic graph, every tight cut is a $3$-cut --- a special case of \cref{every tight cut of a cc2 is a 3-cut}.
Interestingly, \lm\ pairs admit a characterization that involves tight cuts as well as $2$-cuts, but we need a couple of results in order to establish it.
The following lemma, which is a generalization of \cite[Lemma 4.5]{clz25},
relates \lmity\ of pairs in a connected bipartite cubic graph $H$ with \lmity\ of appropriate pairs in its \mbox{$\partial(X)$-contractions}, where $\partial(X)$ is any tight cut.

\begin{lemma}\label{lm pairs across a tight cut - lemma version}
	Let $\partial(X)$ denote a tight cut of a connected bipartite cubic graph $H[A,B]$; adjust notation so that $X^+ \subseteq B$, and let $H_1:=H/\Xbar \rightarrow \xbar$ and $H_2:=H/X \rightarrow x$. Then:
	\begin{enumerate}[label=(\roman*)]
		\item for every $a \in X^-$ and $b \in \Xbar^-$, the pair $(a,b)$ is not \lm\ in $H$,
		\item for every $a \in X^-$ and $b \in X^+$, the pair $(a,b)$ is \lm\ in $H$ if and only if it is \lm\ in $H_1$, and
		\item for every $a \in \Xbar^+$ and $b \in X^+$, the pair $(a,b)$ is \lm\ in $H$ if and only if the pairs $(\xbar, b)$ and $(a,x)$ are \lm\ in $H_1$ and in $H_2$, respectively.
	\end{enumerate}
\end{lemma}
\begin{proof}
	First, let $a \in X^-$  and $b \in \Xbar^-$.
	Suppose that there exists an $(a,b)$-matching $M$.
	Observe that $|\partial_M(X^-)| = |X^-|+2 = |X^+| + 1 > |X^+| = |\partial_{M}(X^+)|$.
	On the other hand, \cref{characterization of tight cuts in bmcgs} implies that $\partial_{M}(X^-) \subseteq \partial_{M}(X^+)$; a contradiction. This proves $(i)$.

	Now, let $a \in X^-$ and $b \in X^+$.
	First, suppose that $(a,b)$ is \lm\ in $H$, and let $M$ denote an $(a,b)$-matching.
	We let $M_1:=M \cap E(H_1)$;
	observe that $d_{M_1}(v)=1$ for each $v \in V(H_1)-\{a,b,\xbar\}$ and that $d_{M_1}(a)=d_{M_1}(b)=3$.
	By counting, we infer that $d_{M_1}(\xbar) = 1$. Ergo, $M_1$ is an $(a,b)$-matching of~$H_1$.
	Conversely, suppose that $(a,b)$ is \lm\ in $H_1$, and let $M_1$ denote an $(a,b)$-matching of~$H_1$.
	Let $M_2$ denote a perfect matching of $H_2$ that contains
	the unique edge in $M_1 \cap \partial_{H_1}(\xbar)$.
	Observe that $M_1 \cup M_2$ is an $(a,b)$-matching of~$H$, and this proves $(ii)$.

	Finally, let $a \in \Xbar^+$ and $b \in X^+$.
	First, suppose that $(a,b)$ is \lm\ in $H$, and let $M$ denote an $(a,b)$-matching.
	We let $M_1:=M \cap E(H_1)$;
	observe that $d_{M_1}(v)=1$ for each $v \in V(H_1)-\{b,\xbar\}$ and that $d_{M_1}(b)=3$.
	By counting, we conclude that $d_{M_1}(\xbar) = 3$. In other words, $M_1$ is an $(\xbar,b)$-matching of $H_1$.
	Thus, $(\xbar,b)$ is \lm\ in $H_1$. An analogous argument shows that $(a,x)$ is \lm\ in~$H_2$.
	Conversely, suppose that $(\xbar,b)$ and $(a,x)$ are \lm\ in $H_1$ and in $H_2$, respectively.
	Let~$M_1$ denote an $(\xbar,b)$-matching of $H_1$ and let $M_2$ denote an $(a,x)$-matching of $H_2$.
	Observe that $M_1 \cup M_2$ is an $(a,b)$-matching of $H$, and this proves $(iii)$.
\end{proof}

The following lemma, which yields \cite[Lemma 4.7]{clz25} if restricted to simple graphs, relates \lmity\ of pairs in a connected bipartite cubic graph $H$ with \lmity\ of appropriate pairs in its marked \mbox{$C$-components}, where $C$ is any $2$-cut of $H$.

\begin{lemma}\label{lm pairs across a 2-cut}
	Let $H[A,B]$ denote a connected bipartite cubic graph that has a $2$-cut $C$, and let $H_1$ and $H_2$ denote its marked $C$-components.
	Then, for any $a \in A$ and $b \in B$, the pair
	$(a,b)$ is \lm\ in $H$ if and only if both $a$ and $b$ belong to the same component of $H-C$ and $(a,b)$ is \lm\ in the corresponding marked $C$-component;
	consequently, $\Rho(H) = \Rho(H_1) \cupdot \Rho(H_2)$ and $\rho(H)=\rho(H_1)+\rho(H_2)$.
\end{lemma}
\begin{proof}
	We let $H_1$ and $H_2$ denote the marked $C$-components, and $e_1$ and $e_2$ their marker edges, respectively. By \cref{cc2 is bipartite iff each marked C-component is bipartite}, both $H_1$ and $H_2$ are bipartite.

	First, suppose that $(a,b)$ is \lm\ in $H$, and let $M$ denote an \mbox{$(a,b)$-matching}.
	Adjust notation so that $a \in V(H_1)$.
	By \paritylemma, $|C \cap M|$ is even.
	If $C\cap M = \emptyset$, let $M_1:=M \cap E(H_1)$;
	otherwise, let $M_1:=(M \cap E(H_1)) + e_1$.
	If $b \notin V(H_1)$, then $d_{M_1}(v)=1$ for each \mbox{vertex $v \in V(H_1)-a$}, and $d_{M_1}(a)=3$; this contradicts \cref{lambda is 0 for bipartite graphs}.
	Therefore, $b \in V(H_1)$. Observe that $M_1$ is \mbox{an $(a,b)$-matching} of $H_1$.

	Now, suppose that $a,b \in V(H_1)$ and $(a,b)$ is \lm\ in $H_1$, and let $M_1$ denote an \mbox{$(a,b)$-matching}.
	Let $M_2$~and~$N_2$ denote perfect matchings of $H_2$ containing $e_2$ and not containing $e_2$, respectively.
	If~$e_1 \in M_1$, let $M:=(M_1-e_1) + (M_2-e_2) + C$;
	otherwise, let $M:=M_1 + N_2$.
	In either case, the reader may verify that $M$ is an $(a,b)$-matching of $H$.
\end{proof}

We are now ready to prove our characterization of \lm\ pairs;
this, along with \cref{in a simple brace every pair is lm},
implies \cite[Proposition 4.4]{clz25}.

\begin{lemma}
	\label{characterization of pairs that are not lm}
Let $H[A,B]$ denote a connected bipartite cubic graph, and let $a \in A$ and $b \in B$. Then, $(a,b)$ is not \lm\ if and only if precisely one of the following holds: $(i)$ there exists a (unique) edge joining~$a$~and~$b$ and it participates in some $2$-cut, or $(ii)$ $a$ and $b$ are nonadjacent and there exists a (nontrivial) tight cut~$\partial(X)$ such that $a \in X^-$ and $b \in \overline{X}^-$.
\end{lemma}
\begin{proof}
	Observe that the reverse implication holds by \cref{lm pairs across a 2-cut} and \cref{lm pairs across a tight cut - lemma version} $(i)$.
	Conversely, assume that~$(a,b)$ is not \lm\ in $H$, whence $H$ is not $\Theta$.
	First suppose that at least one of~$a$~and~$b$ has precisely two neighbors; adjust notation so that $a$ has neighbors $b_1$ and $b_2$, and there are two edges joining $a$ and $b_2$.
	If $a$ and $b$ are nonadjacent then $\partial(X)$, where $X:=\{a,b_1,b_2\}$, satisfies $(ii)$.
	Now suppose that $a$ and $b$ are adjacent, and consider the $2$-cut $C:=\partial(\{a,b_2\})$.
	By the reverse implication of \cref{lm pairs across a 2-cut}, $(a,b_2)$ is \lm{}; ergo, $b=b_1$ and $(i)$ holds.

	Next, suppose that each of $a$ and $b$ has three distinct neighbors.
	Hence, the graph $J:=H-a-b-N_H(a)-N_H(b)$ is not matchable;
	consequently, by Hall's Theorem, there exists $Y' \subseteq V(J) \cap A$ such that $|N_J(Y')| \leqslant |Y'|-1$.
	We let $Y:=Y' + a$.
	By definition of $J$, the vertex~$b$ has no neighbors in $Y'$ in the graph $H$; consequently, $N_H(Y) = N_J(Y') \cupdot N_H(a)$. 
	Therefore, $|N_H(Y)| = |N_J(Y')| + 3 \leqslant (|Y'|-1)+3 = |Y|+1$;
	by \cref{characterization of bmcgs} $(iii)$, $|N_H(Y)|=|Y|+1$.
	Observe that $\partial_H(X)$, where $X:=Y \cup N_H(Y)$, is a tight cut as well as a $3$-cut.
	If $a$ and $b$ are nonadjacent, then~$\partial_H(X)$ is the desired tight cut that satisfies $(ii)$.
	Otherwise, $b \in X$ and $|\partial_H(b) \cap \partial_H(X)| = 2$;
	note that $\partial_H(X-b)$ is the desired $2$-cut that satisfies $(i)$.
\end{proof}

The below intermediate result is implied by the above.
Note that it immediately yields a linear bound of $\frac{3n}{2}$ on $\rho$ for \con3 bipartite cubic graphs, but we shall find it useful to prove our aforementioned stronger lower bounds.

\begin{corollary}\label{each edge is lm unless it participates in a 2-cut}
	In a connected bipartite cubic graph $H[A,B]$, each edge $e:=ab$ satisfies precisely one of the following: $(i)$ either $e$ participates in a $2$-cut, or otherwise $(ii)$ the pair $(a,b)$ is \lm. \qed
\end{corollary}

We now redirect our focus back to proving lower bounds on $\rho$.
The following is simply a restatement of \cref{lm pairs across a tight cut - lemma version} in a more compact form.

\begin{theorem}\label{lm pairs across a tight cut - set version}
	For any tight cut $\partial(X)$ of a connected bipartite cubic graph $H[A,B]$, adjusting notation so that $X^+ \subseteq B$, the following holds:
	\[\Rho(H) = \big\{ (a,b) \in \Rho(H_1) \mid a \neq \xbar\big\}~\cupdot~\big\{(a,b) \in \Rho(H_2) \mid b \neq x\big\}~\cupdot \] \[\big\{ (a,b) \mid  (\xbar,b) \in \Rho(H_1) \text{ and } (a,x) \in \Rho(H_2) \big\} \]
	where $H_1:=H/\Xbar \rightarrow \xbar$ and $H_2:=H/X \rightarrow x$. \qed
\end{theorem}

Using the above theorem, one may easily compute the number of \lm\ pairs in~$H$ using information pertaining to \lm\ pairs of its $\partial(X)$-contractions. To this end, for a vertex $u$ of $H$, we let $\numlmpartners{H}{u}$ denote the number of vertices $v \in V(H)$ such that the pair $(u,v)$ is \lm\ in $H$. The statement below is an immediate consequence of the above.

\begin{corollary}\label{lm pairs across a tight cut - count version}
	For any tight cut $\partial(X)$ of a connected bipartite cubic graph $H$, the following holds:
	\[\rho(H)~=~\rho(H_1)-\numlmpartners{H_1}{\xbar}~+~\rho(H_2)-\numlmpartners{H_2}{x}~+~\numlmpartners{H_1}{\xbar}\numlmpartners{H_2}{x}\]
	where $H_1:=H/\Xbar \rightarrow \xbar$ and $H_2:=H/X \rightarrow x$. \qed
\end{corollary}

By \cref{in a simple brace every pair is lm}, each simple cubic brace satisfies $\rho = \b + 3b' - 3$; we let $\braces$ denote the family that comprises them. Now, we define another family $\K$ of \bcc3s\ whose members satisfy the bound with equality (as proven below): $\K$ contains $K_{3,3}$, and if $H \in \K$ then $(H \odot K_{3,3})_v \in \K$ for any $v \in V(H)$ such that $\numlmpartners{H}{v}=3$. For instance, the graph $K_{3,3} \odot K_{3,3}$, shown in \cref{fig: k33 splice k33}, is a member if $\K$.

\begin{proposition}\label{rho lower bound 3-conn - easy direction}
	The family $\K$ is infinite, and each of its members, say $H$, is a \bcc3 that satisfies $\rho(H) = \b(H) + 3b'(H) - 3 = 3n(H) -9$.
\end{proposition}
\begin{proof}
	Clearly, $H$ is bipartite and cubic, and the fact that $H$ is \con3 follows from \cref{splicing of two 3-conn cubic graphs is 3-conn}.
	To prove that~$\K$ is an infinite family, the reader may observe that it suffices to show that each $H \in \K$ has a vertex~$v$ such that $\numlmpartners{H}{v}=3$; additionally, we need to establish that $\rho(H) = \b(H) + 3b'(H) - 3 = 3n(H) -9$.
	We proceed by induction on the order.

	If $H=K_{3,3}$, the reader may easily verify that the desired conclusions hold.
	Now, suppose that $H$ is not $K_{3,3}$.
	By definition of $\K$, note that $H$ has a tight cut $\partial(X)$ such that $H_1:=H/\Xbar \rightarrow \xbar$ belongs to $\K$ and $\numlmpartners{H_1}{\xbar}=3$, whereas $H_2:=H/X \rightarrow x$ is $K_{3,3}$.
	By \cref{lm pairs across a tight cut - lemma version} $(i)$, for any~$a \in \Xbar^-$ and~$b \in X^-$, the pair~$(a,b)$ is not \lm\ in $H$; consequently, $\numlmpartners{H}{a} \leqslant |\Xbar^+| = 3$.
	However, since $H$ is \con3, by invoking \cref{each edge is lm unless it participates in a 2-cut}, we infer that $\numlmpartners{H}{a} = 3$.
	As noted earlier, this proves that~$\K$ is an infinite family.
	Now, by \cref{lm pairs across a tight cut - count version}:
	\begin{equation*}
	\begin{split}
	\begin{array}[b]{ c @{} c @{} c @{} c @{} c @{} c @{} c @{} c @{} c @{} c @{} c }
		\rho(H)~
		&~=~& ~\rho(H_1)~
		&~-~& \numlmpartners{H_1}{\xbar}
		&~+~& \rho(H_2)
		&~-~& \numlmpartners{H_2}{x}
		&~+~& \numlmpartners{H_1}{\xbar} \numlmpartners{H_2}{x} \\
		&~=~& \rho(H_1) &~-~& 3  &~+~& 9 &~-~& 3  &~+~& 3 \,\cdotp 3 \\
		&~=~& \rho(H_1) &~+~& 12 \\
	\end{array}
	\end{split}
	\end{equation*}

	Observe that $n(H) = n(H_1) + 4$, $b'(H)=b'(H_1)+1$ and $\b(H)=\b(H_1)+9$.
	By the induction hypothesis, $\rho(H_1) = \b(H_1) + 3b'(H_1) - 3$; consequently, $\rho(H)=(\b(H_1)+3b'(H_1)-3)+12=\b(H)+3b'(H)-3$.
	Also, by the induction hypothesis, $\rho(H_1) =  3n(H_1) -9$; consequently, $\rho(H) = (3n(H_1)-9)+12=3n(H)-9$.
	This completes the proof.
\end{proof}

We now state and prove a technical lemma that we shall find useful in proving our lower bounds.

\begin{lemma}\label{technical lemma - lower bound on pq-p-q}
	Every pair of integers $p$ and $q$, each of which is at least three, satisfies $pq-p-q~\geqslant~3$, and equality holds if and only if $p=q=3$. \qed
\end{lemma}
\begin{proof}
	Clearly, both $(p-3)$ and $(q-3)$ are at least zero.
	Therefore, the lemma follows from the observation that $pq-p-q=(p-3)(q-3)+2(p-3)+2(q-3)+3$. 
\end{proof}

Now, we are ready to establish our lower bounds on $\rho$ for \bcc3s (stated in \cref{rho lower bound 3-conn - intro}), and characterize the corresponding tight examples.

\begin{theorem}\label{rho lower bound 3-conn}
	Every \bcc3\ $H$ satisfies:
	\[ \rho(H) ~\geqslant~ \b(H) + 3b'(H) - 3 ~\geqslant~ 3n(H) -9\]
	Furthermore, (i) the first inequality holds with equality if and only if $H \in \braces \cup \K$, and (ii)~equality holds everywhere if and only if $H \in \K$.
\end{theorem}
\begin{proof}
	By \cref{in a simple brace every pair is lm,rho lower bound 3-conn - easy direction}, it suffices to prove that the inequalities hold, and that the forward implications of $(i)$ and $(ii)$ hold.
	We proceed by induction on the order. If $H=\Theta$, the conclusion may be verified easily.

	Now, suppose that $H \in \braces$; by \cref{in a simple brace every pair is lm}, \mbox{$\rho(H) = \b(H)=\b(H)+3b'(H)-3$}.
	For any real number~$n$, note that $n \geqslant 6$ if and only if $\ds \left(\frac{n}{2}\right)^2 \geqslant 3n-9$, and equality holds if and only if $n=6$. Consequently, $\rho(H)=\b(H) \geqslant 3n(H)-9$ and if equality holds, then $H = K_{3,3}$, which belongs to $\K$. In summary, the desired conclusions hold.

	Henceforth, suppose that $H \notin \braces$.
	We invite the reader to observe that it suffices to prove that the inequalities hold, and that the forward implication of $(i)$ holds.
	As $H$ is~not a brace, let $\partial(X)$ be any nontrivial tight cut, and let $H_1:=H/\Xbar \rightarrow \xbar$ and $H_2:=H/X \rightarrow x$. By \cref{characterization of tight cuts in bmcgs,3-cut contractions preserve connectivity,every tight cut of a cc2 is a 3-cut}, both $H_1$ and $H_2$ are \bcc3s; by the induction hypothesis,
	$\rho(H_i) \geqslant \b(H_i) + 3b'(H_i) - 3 \geqslant 3n(H_i)-9$ for each $i \in \{1,2\}$.

	By \cref{lm pairs across a tight cut - count version}, $\rho(H) = \rho(H_1) + \rho(H_2) + \numlmpartners{H_1}{\xbar}\numlmpartners{H_2}{x} -\numlmpartners{H_1}{\xbar} -\numlmpartners{H_2}{x}$.
	By \cref{each edge is lm unless it participates in a 2-cut}, both $\numlmpartners{H_1}{\xbar}$ and $\numlmpartners{H_2}{x}$ are at least three; consequently, by \cref{technical lemma - lower bound on pq-p-q}, $\rho(H) \geqslant \rho(H_1) + \rho(H_2) + 3$ and equality holds if and only if $\numlmpartners{H_1}{\xbar} = \numlmpartners{H_2}{x}=3$. Now, putting all of this together:
	\begin{equation*}
		\begin{split}
			\begin{array}[b]{ c @{} c @{} c @{} c @{} c @{} c @{} c @{} c @{} c @{} c @{} c @{} c @{} c @{} c @{} c @{} c @{} c}
				\rho(H) &~\geqslant~& \rho(H_1) &~+~& \rho(H_2) &~+~& 3 \\
				&~\geqslant~& (\b(H_1) + 3b'(H_1) - 3) &~+~& (\b(H_2) + 3b'(H_2) - 3) &~+~& 3 \\
				&~=~& \underbrace{(\b(H_1)+\b(H_2))} &~+~& 3\underbrace{(b'(H_1)+b'(H_2))} &~-~& 3 \\
				&~=~& \b(H) &~+~& 3b'(H) &~-~& 3
			\end{array}
		\end{split}
	\end{equation*}
	as well as:
	\begin{equation*}
		\begin{split}
			\begin{array}[b]{ c @{} c @{} c @{} c @{} c @{} c @{} c @{} c @{} c @{} c @{} c @{} c @{} c @{} c @{} c @{} c @{} c}
				\b(H) + 3b'(H) - 3 &~~=~~& (\b(H_1)+3b'(H_1)-3) &~~+~~& (\b(H_2) + 3b'(H_2) - 3) &~+~& 3 \\
				&~~\geqslant~~& (3n(H_1) - 9) &~~+~~& (3n(H_2) - 9) &~+~& 3 \\
				&~~=~~& 3\underbrace{(n(H_1)+n(H_2)-2)} &~+~& 6-9-9+3~~~~~~~~~ \\
				&~~=~~& 3n(H) &~-~& \hspace{-10.5mm} 9
			\end{array}
		\end{split}
	\end{equation*}

	In summary, we have proved the inequalities, and as noted earlier, it remains to prove the forward implication of $(i)$. To this end, we choose $\partial(X)$ to be a peripheral tight cut (of~$H$) so that $H_2$ is a brace. Assume that $\rho(H) = \b(H)+3b'(H)-3$. It follows from the earlier chain of inequalities that $\rho(H) = \rho(H_1) + \rho(H_2) + 3$ and $\rho(H_1)=\b(H_1)+3b'(H_1)-3$. The latter equation implies, by the induction hypothesis, that $H_1 \in \braces \cup \K$, whereas the former one implies, as per our earlier discussion, that $\numlmpartners{H_1}{\xbar} = \numlmpartners{H_2}{x}=3$.
	Note that for any~$J \in \braces - K_{3,3}$ and $v \in V(J)$, \cref{in a simple brace every pair is lm} implies that $\numlmpartners{J}{v} \geqslant 4$. Applying this observation to $H_1$ as well as~$H_2$, the reader may infer that $H_1 \in \K$ and~$H_2=K_{3,3}$; consequently, $H \in \K$.

	This completes the proof of \cref{rho lower bound 3-conn}.
\end{proof}

In order to prove our extension of the above theorem to all connected bipartite cubic graphs, we shall use $2$-cuts as an induction tool.
We proceed to define families $\K'$ and $\L$ of connected bipartite cubic graphs that appear in the characterizations of tight examples for our lower bounds. The family~$\K'$ comprises those graphs whose \mbox{$3$-connected} pieces belong to $\K \cup \{\Theta\}$, whereas $\L$ comprises those graphs whose \mbox{$3$-connected} pieces belong to $\braces \cup \K \cup \{\Theta\}$.
For instance, the graph shown in \cref{fig: k33 glue theta glue k33} is a member of both $\K'$ and $\L$, whereas the graph shown in \cref{fig: script l prime example} is a member of $\L$ but not of $\K'$.

\fig{A member of $\L$ but not of $\K'$}{fig: script l prime example}{\begin{tikzpicture}[scale=1.35]

\tikzmath{
	\one = 0.365;
	\two = 0.825;
}


\node[vtx] (a1) at (-\two,  \two) {};
\node[vtx-white] (a2) at ( \two,  \two) {};
\node[vtx] (a3) at ( \two, -\two) {};
\node[vtx-white] (a4) at (-\two, -\two) {};

\node[vtx-white] (b1) at (-\one,  \one) {};
\node[vtx] (b2) at ( \one,  \one) {};
\node[vtx-white] (b3) at ( \one, -\one) {};
\node[vtx] (b4) at (-\one, -\one) {};

\foreach \i in {1,2,3}{
	\node[vtx-white] (x\i) at (1.25*\i+0.75,-0.825) {};
	\node[vtx] (y\i) at (1.25*\i+0.75,0.825) {};
}


\foreach \i in {1,2,3} {
	\foreach \j in {1,2,3} {
		\ifnum \i=1
			\ifnum \j=1
				\draw (x\i) -- (a3);
				\draw (y\j) -- (a2);
			\else
				\draw (x\i) -- (y\j);
			\fi
		\else
			\draw (x\i) -- (y\j);
		\fi
	}
}

\draw (a3) -- (a4) -- (a1) -- (a2);
\draw (b1) -- (b2) -- (b3) -- (b4) -- (b1);
\foreach \i in {1,...,4}
	\draw (b\i) -- (a\i);

\end{tikzpicture}}

The following is an easy consequence of \cref{rho lower bound 3-conn - easy direction}, \cref{lm pairs across a 2-cut}, \cref{bstar across even 2-cut} and \cref{theta thetabar and n-nonbip over a 2 cut}.

\begin{proposition}\label{rho lower bound 2-conn - easy direction}
	Every member $H$ of $\L$ satisfies $\rho(H) = \b(H) + 3b'(H) - 3\thetabar(H)+\theta(H)$, and every member~$H$ of $\K'$ satisfies $\rho(H) = 3n(H) -9\thetabar(H)-5\theta(H)$. \qed
\end{proposition}


We remark that one may equivalently define either of the above families in a recursive manner.
To this end, we first define an operation.
The (\con2 cubic) graph obtained by \emph{gluing} disjoint \cc2s $G_1$ and $G_2$ at specified edges $e_1:=u_1v_1$ and $e_2:=u_2v_2$, respectively,
is the one obtained from the union of $G_1-e_1$ and $G_2-e_2$ by adding the edges $u_1u_2$ and $v_1v_2$.
Observe that, depending on the labels of the ends of $e_1$ and $e_2$, one may obtain two such graphs.
\cref{fig: k4 glue k33,fig: script l prime example} show examples.
We may now define $\K'$ as follows: firstly, $\K \cup \Theta \subseteq \K'$, and secondly, if $G_1,G_2 \in \K'$ then any graph obtained by gluing them also belongs to $\K'$.

Finally, we prove the lower bounds on $\rho$ that apply to all connected bipartite cubic graphs (mentioned in \cref{rho lower bound 2-conn - intro}), and characterize the corresponding tight examples.
It is in fact an easy consequence of \cref{rho lower bound 3-conn}, \cref{lm pairs across a 2-cut}, \cref{bstar across even 2-cut} and \cref{theta thetabar and n-nonbip over a 2 cut}; however, we choose to include its proof.

\begin{corollary}
	Every connected bipartite cubic graph $H$ satisfies:
	\[ \rho(H) ~\geqslant~ \b(H) + 3b'(H) - 3\thetabar(H)+\theta(H) ~\geqslant~ 3n(H) -9\thetabar(H)-5\theta(H)\]
	Furthermore, (i) the first inequality holds with equality if and only if $H \in \L$, and (ii)~equality holds everywhere if and only if $H \in \K'$.
\end{corollary}
\begin{proof}
	By \cref{rho lower bound 2-conn - easy direction}, it suffices to prove that the inequalities hold, and that the forward implications of $(i)$ and $(ii)$ hold. We proceed by induction on the order. If $H$ is \mbox{$3$-connected}, by invoking \cref{rho lower bound 3-conn}, the reader may verify that the desired conclusions hold. Now, suppose that $H$ is not \mbox{$3$-connected}; let $C$ denote a $2$-cut, and let $H_1$~and~$H_2$ denote the marked \mbox{$C$-components}. Clearly, $H_1$ and $H_2$ are connected and cubic; by~\cref{cc2 is bipartite iff each marked C-component is bipartite}, they are bipartite as well. By the induction hypothesis, for each $i \in \{1,2\}$:
	\begin{equation}\label{P(H_i) >= beta prime bound for i = 1 and 2}
		\rho(H_i) ~\geqslant~ \b(H_i) + 3b'(H_i) - 3\thetabar(H_i)+\theta(H_i) ~\geqslant~ 3n(H_i) -9\thetabar(H_i)-5\theta(H_i)
	\end{equation}
	By \cref{lm pairs across a 2-cut}:
	\begin{equation}\label{rho(H) is the sum of that of H_1 and H_2}
		\rho(H)=\rho(H_1)+\rho(H_2)
	\end{equation}
	By \cref{bstar across even 2-cut} and \cref{theta thetabar and n-nonbip over a 2 cut}:
	\begin{equation}\label{beta prime bound of H is the sum of that of H_1 and H_2}
		\begin{split}
			\b(H) + 3b'(H) - 3\thetabar(H)+\theta(H) ~=~
			\begin{array}{ c @{} c @{} c @{} c @{} c @{} c @{} c @{} c @{} c @{} c @{} c @{} c @{} c @{} c @{} c @{} c @{} c}
				&& \b(H_1) + 3b'(H_1) - 3\thetabar(H_1)+\theta(H_1) \\
				&+~& \b(H_2) + 3b'(H_2) - 3\thetabar(H_2)+\theta(H_2)  \\
			\end{array}
		\end{split}
	\end{equation}
	As $n(H)=n(H_1)+n(H_2)$, by \cref{theta thetabar and n-nonbip over a 2 cut}:
	\begin{equation}\label{linear bound of H is the sum of that of H_1 and H_2}
		\begin{split}
			n(H) - 9\thetabar(H)-5\theta(H) ~=~
			\begin{array}{ c @{} c @{} c @{} c @{} c @{} c @{} c @{} c @{} c @{} c @{} c @{} c @{} c @{} c @{} c @{} c @{} c}
				&& n(H_1) - 9\thetabar(H_1)-5\theta(H_1) \\
				&+~& n(H_2) - 9\thetabar(H_2)-5\theta(H_2)  \\
			\end{array}
		\end{split}
	\end{equation}
	By combining \cref{P(H_i) >= beta prime bound for i = 1 and 2,rho(H) is the sum of that of H_1 and H_2,beta prime bound of H is the sum of that of H_1 and H_2,linear bound of H is the sum of that of H_1 and H_2} above, we may conclude that:
	\[ \rho(H) ~\geqslant~ \b(H) + 3b'(H) - 3\thetabar(H)+\theta(H) ~\geqslant~ 3n(H) -9\thetabar(H)-5\theta(H)\]
	Now, if $\rho(H) = \b(H) + 3b'(H) - 3\thetabar(H)+\theta(H)$ then $\rho(H_i)=\b(H_i) + 3b'(H_i) - 3\thetabar(H_i)+\theta(H_i)$ for each $i \in \{1,2\}$; hence, by the induction hypothesis, $H_1, H_2 \in \L$; by definition, $H \in \L$.
	Likewise, if $\rho(H) = 3n(H) -9\thetabar(H)-5\theta(H)$ then $\rho(H_i)=3n(H_i) -9\thetabar(H_i)-5\theta(H_i)$ for~each $i \in \{1,2\}$; hence, by the induction hypothesis, $H_1, H_2 \in \K'$; by definition, $H \in \K'$.
\end{proof}

\section{\lm\ vertices in cubic graphs}\label{sec: lm vertices in cubic graphs}
In order to establish the bound $\lambda \geqslant \B$, we first deal with \cc3s as in the preceding section.
To this end, we use nontrivial barriers as our induction tool;
see \cref{a 2 conn cubic graph is not theta or brick iff there is a nontrivial barrier}.
We shall first look at some matching-theoretic concepts related to nontrivial barriers that apply to all matching covered graphs and play an important role in our work.
Observe that if $B$ is any barrier of a \mcg\ $G$ and if $J$ is any (odd) component of $G-B$, then $C:=\partial(V(J))$ is a tight cut; such cuts are called \emph{barrier cuts}, and the \mbox{$C$-contraction}
$G/\overline{V(J)}$ is called a \emph{barrier fragment} of $G$ with respect to $B$, or simply a \emph{$B$-fragment} of~$G$.
Clearly, the number of $B$-fragments is $|B|$.
Additionally, the bipartite (\mc) graph $H[A,B]$ obtained from~$G$ by shrinking each (odd) component of $G-B$ into a single vertex is called the \emph{core} of $G$ with respect to $B$. 
Since each vertex $a \in A$ is obtained by shrinking some odd component $J$, which in turn corresponds to a \mbox{$B$-fragment} $G/\overline{V(J)}$, there is a natural correspondence between the vertices in $A$ and the $B$-fragments, that we shall find useful. Below is an example wherein the correspondence is shown using the same subscripts.

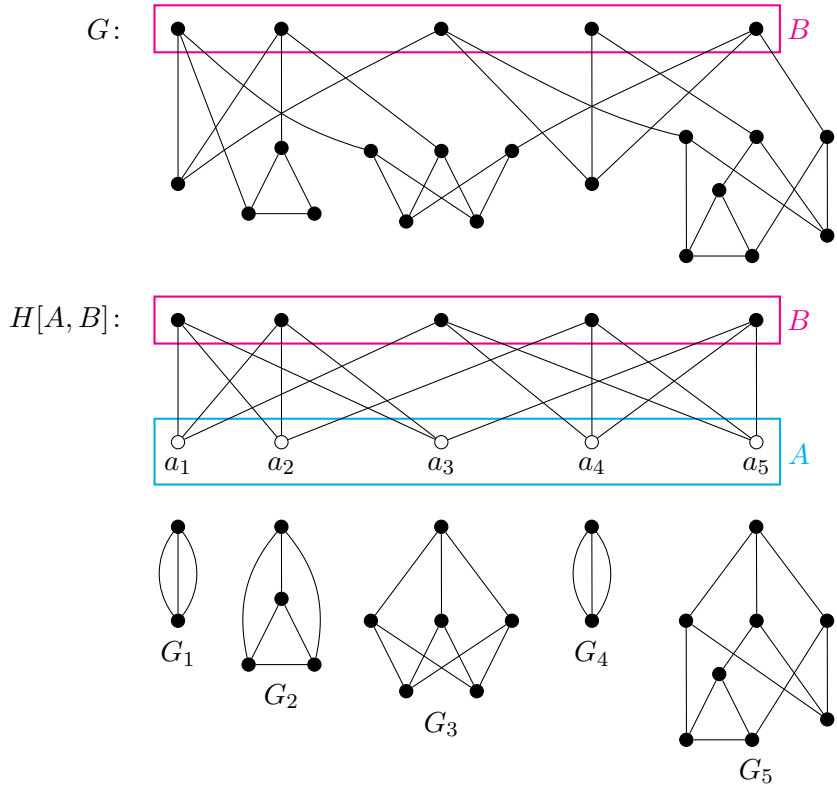
\fig{\mbox{A barrier $B$ of a \cc3 $G$, the corresponding core $H$, and the $B$-fragments $G_i$'s}}{fig: barriers and fragments example}{\begin{tikzpicture}[xscale=1.3,yscale=1.4]


\tikzmath{
	\height = 0.7;
	\halfsidelength = sin(30)*\height;
	\sidelength = 2*\halfsidelength;
}

\node[vtx] (b1) at (1,0) {};
\node[vtx] (b2) at (2.1,0) {};
\node[vtx] (b3) at (3.8,0) {};
\node[vtx] (b4) at (5.4,0) {};
\node[vtx] (b5) at (7.15,0) {};

\node[anchor=east] at ($(b1)+(-0.5,0)$) {$G\!:$};
\draw[magenta,thick] ($(b1)+(-0.25,0.25)$) rectangle ($(b5)+(0.25,-0.25)$);
\node[magenta] at ($(b5)+(0.45,0)$) {$B$};

\coordinate (a1) at ($(b1)+(0,-1.65)$);
\coordinate (a2) at ($(b2)+(0,-1.5)$);
\coordinate (a3) at ($(b3)+(0,-1.3)$);
\coordinate (a4) at ($(b4)+(0,-1.65)$);
\coordinate (a5) at ($(b5)+(0,-1.15)$);

\node[vtx] (p1) at (a1) {};

\node[vtx] (q2) at ($(a2)+(0,\height/3)$) {};
\node[vtx] (q1) at ($(a2)+(-\halfsidelength, -2*\height/3)$) {};
\node[vtx] (q3) at ($(a2)+(+\halfsidelength, -2*\height/3)$) {};

\tikzmath{ \horizgap = 0.75; \vertgap = 0.75; }
\node[vtx] (r2) at (a3) {};
\node[vtx] (r1) at ($(a3)+(-\horizgap,0)$) {};
\node[vtx] (r3) at ($(a3)+( \horizgap,0)$) {};
\node[vtx] (r4) at ($(r1)!0.5!(r2)+(0,-\vertgap)$) {};
\node[vtx] (r5) at ($(r2)!0.5!(r3)+(0,-\vertgap)$) {};

\node[vtx] (s1) at (a4) {};

\tikzmath{ \horizgap = 0.75; \height = 0.7; }
\node[vtx] (t2) at (a5) {};
\node[vtx] (t1) at ($(a5)+(-\horizgap,0)$) {};
\node[vtx] (t3) at ($(a5)+( \horizgap,0)$) {};
\coordinate (x) at ($(t1)+(\halfsidelength,-0.8)$) {};
\node[vtx] (t5) at ($(x)+(0,\height/3)$) {};
\node[vtx] (t4) at ($(x)+(-\halfsidelength, -2*\height/3)$) {};
\node[vtx] (t6) at ($(x)+(+\halfsidelength, -2*\height/3)$) {};
\node[vtx] (t7) at ($(t3)+(0,-1.05)$) {};

\draw (p1) -- (b1);
\draw (p1) -- (b2);
\draw plot [smooth,tension=1] coordinates {(p1) ($(b2)!0.6!(a2)$) (b3)};

\draw (q1) -- (b1);
\draw (q2) -- (b2);
\draw plot [smooth,tension=1] coordinates {(q3) ($(b3)!0.4!(r1)+(0.05,-0.1)$) (b4)};
\draw (q1) -- (q2) -- (q3) -- (q1);

\draw plot [smooth,tension=1] coordinates {(r1) ($(b2)!0.6!(a2)$) (b1)};
\draw (r2) -- (b2);
\draw plot [smooth,tension=1] coordinates {(r3) ($(b4)!0.5!(a4)$) (b5)};
\draw (r4) -- (r1);
\draw (r4) -- (r2);
\draw (r4) -- (r3);
\draw (r5) -- (r1);
\draw (r5) -- (r2);
\draw (r5) -- (r3);

\draw (s1) -- (b3);
\draw (s1) -- (b4);
\draw (s1) -- (b5);

\draw plot [smooth,tension=1] coordinates {(t1) ($(b4)!0.5!(a4)$) (b3)};
\draw (t2) -- (b4);
\draw (t3) -- (b5);
\draw (t4) -- (t5) -- (t6) -- (t4);
\draw (t1) -- (t4);
\draw (t2) -- (t5);
\draw (t3) -- (t6);
\draw (t1) -- (t7);
\draw (t2) -- (t7);
\draw (t3) -- (t7);

\begin{scope}[shift={(0,-3.1)}]
	\node[vtx] (b1) at (1,0) {};
	\node[vtx] (b2) at (2.1,0) {};
	\node[vtx] (b3) at (3.8,0) {};
	\node[vtx] (b4) at (5.4,0) {};
	\node[vtx] (b5) at (7.15,0) {};

	\node[vtx-white] (a1) at ($(b1)+(0,-1.3)$) {};
	\node[vtx-white] (a2) at ($(b2)+(0,-1.3)$) {};
	\node[vtx-white] (a3) at ($(b3)+(0,-1.3)$) {};
	\node[vtx-white] (a4) at ($(b4)+(0,-1.3)$) {};
	\node[vtx-white] (a5) at ($(b5)+(0,-1.3)$) {};

	\node[anchor=east] at ($(b1)+(-0.5,0)$) {$H[A,B]\!:$};
	\draw[magenta,thick] ($(b1)+(-0.25,0.25)$) rectangle ($(b5)+(0.25,-0.25)$);
	\node[magenta] at ($(b5)+(0.45,0)$) {$B$};
	\draw[cyan,thick] ($(a1)+(-0.25,0.25)$) rectangle ($(a5)+(0.25,-0.45)$);
	\node[cyan] at ($(a5)+(0.45,-0.125)$) {$A$};
	\foreach \i in {1,...,5}{
		\node at ($(a\i)+(0,-0.25)$) {$a_\i$};
	}

	\draw (a1) -- (b1);
	\draw (a1) -- (b2);
	\draw (a1) -- (b3);

	\draw (a2) -- (b1);
	\draw (a2) -- (b2);
	\draw (a2) -- (b4);

	\draw (a3) -- (b1);
	\draw (a3) -- (b2);
	\draw plot [smooth,tension=1] coordinates {(a3) ($(b4)!0.5!(a4)$) (b5)};

	\draw (a4) -- (b3);
	\draw (a4) -- (b4);
	\draw (a4) -- (b5);

	\draw plot [smooth,tension=1] coordinates {(a5) ($(b4)!0.5!(a4)$) (b3)};
	\draw (a5) -- (b4);
	\draw (a5) -- (b5);

	\node[vtx-white] (a3) at ($(b3)+(0,-1.3)$) {};
	\node[vtx-white] (a5) at ($(b5)+(0,-1.3)$) {};
\end{scope}

\begin{scope}[shift={(0,-5.5)}]
	\node[vtx] (b1) at (1,0.2) {};
	\node[vtx] (b2) at (2.1,0.2) {};
	\node[vtx] (b3) at (3.8,0.2) {};
	\node[vtx] (b4) at (5.4,0.2) {};
	\node[vtx] (b5) at (7.15,0.2) {};

	\coordinate (a1) at ($(b1)+(0,-1)$);
	\coordinate (a2) at ($(b2)+(0,-1)$);
	\coordinate (a3) at ($(b3)+(0,-1)$);
	\coordinate (a4) at ($(b4)+(0,-1)$);
	\coordinate (a5) at ($(b5)+(0,-1)$);

	\node[vtx] (p1) at (a1) {};

	\node[vtx] (q2) at ($(a2)+(0,\height/3)$) {};
	\node[vtx] (q1) at ($(a2)+(-\halfsidelength, -2*\height/3)$) {};
	\node[vtx] (q3) at ($(a2)+(+\halfsidelength, -2*\height/3)$) {};

	\tikzmath{ \horizgap = 0.75; \vertgap = 0.75; }
	\node[vtx] (r2) at (a3) {};
	\node[vtx] (r1) at ($(a3)+(-\horizgap,0)$) {};
	\node[vtx] (r3) at ($(a3)+( \horizgap,0)$) {};
	\node[vtx] (r4) at ($(r1)!0.5!(r2)+(0,-\vertgap)$) {};
	\node[vtx] (r5) at ($(r2)!0.5!(r3)+(0,-\vertgap)$) {};

	\node[vtx] (s1) at (a4) {};

	\tikzmath{ \horizgap = 0.75; \height = 0.7; }
	\node[vtx] (t2) at (a5) {};
	\node[vtx] (t1) at ($(a5)+(-\horizgap,0)$) {};
	\node[vtx] (t3) at ($(a5)+( \horizgap,0)$) {};
	\coordinate (x) at ($(t1)+(\halfsidelength,-0.8)$) {};
	\node[vtx] (t5) at ($(x)+(0,\height/3)$) {};
	\node[vtx] (t4) at ($(x)+(-\halfsidelength, -2*\height/3)$) {};
	\node[vtx] (t6) at ($(x)+(+\halfsidelength, -2*\height/3)$) {};
	\node[vtx] (t7) at ($(t3)+(0,-1.05)$) {};

	\draw (p1) -- (b1);
	\draw plot [smooth,tension=1.125] coordinates {(p1) ($(p1)!0.5!(b1)+(0.2,0)$) (b1)};
	\draw plot [smooth,tension=1.125] coordinates {(p1) ($(p1)!0.5!(b1)-(0.2,0)$) (b1)};

	\draw (q2) -- (b2);
	\draw (q1) to[in=180+50,out=90+12] (b2);
	\draw (q3) to[in=-50,out=90-12] (b2);
	\draw (q1) -- (q2) -- (q3) -- (q1);

	\draw (r1) -- (b3);
	\draw (r2) -- (b3);
	\draw (r3) -- (b3);
	\draw (r4) -- (r1);
	\draw (r4) -- (r2);
	\draw (r4) -- (r3);
	\draw (r5) -- (r1);
	\draw (r5) -- (r2);
	\draw (r5) -- (r3);

	\draw (s1) -- (b4);
	\draw plot [smooth,tension=1.125] coordinates {(s1) ($(s1)!0.5!(b4)+(0.2,0)$) (b4)};
	\draw plot [smooth,tension=1.125] coordinates {(s1) ($(s1)!0.5!(b4)-(0.2,0)$) (b4)};

	\draw (t1) -- (b5);
	\draw (t2) -- (b5);
	\draw (t3) -- (b5);
	\draw (t4) -- (t5) -- (t6) -- (t4);
	\draw (t1) -- (t4);
	\draw (t2) -- (t5);
	\draw (t3) -- (t6);
	\draw (t1) -- (t7);
	\draw (t2) -- (t7);
	\draw (t3) -- (t7);
\end{scope}

\node at ($(a1)+(0,-0.35)$) {$G_1$};
\node at ($(a2)+(0,-0.8)$) {$G_2$};
\node at ($(a3)+(0,-1.1)$) {$G_3$};
\node at ($(a4)+(0,-0.35)$) {$G_4$};
\node at ($(a5)+(0,-1.6)$) {$G_5$};

\node at (8.65,0) {};

\end{tikzpicture}}

The following is easy to see.

\begin{lemma}\label{B* over a barrier}
	Let $B$ denote a nontrivial barrier of a \mcg\ $G$, let $H$ denote the corresponding core, and let $G_1,G_2,\dots,G_{r}$ denote the $B$-fragments of order four or more. Then:
	\begin{eqnqed}
		\bstar(G)~=~\bstar(H) ~\cupdot~ \left(\, \bigcupdot_{i=1}^{r}\bstar(G_i) \right)
	\end{eqnqed}
\end{lemma}

The following observation is an immediate consequence of the above.

\begin{corollary}\label{beta over a barrier}
	Let $B$ denote a barrier of a \mcg\ $G$, let $H$ denote the corresponding core, and let $G_1,G_2,\dots,G_{|B|}$ denote the $B$-fragments. Then:
	\begin{eqnqed}
		\B(G)~=~\sum_{i=1}^{|B|}\B(G_i)
	\end{eqnqed}
\end{corollary}

Using a straightforward counting argument, we observe that in a \cc2, every barrier cut is a $3$-cut --- a special case of \cref{every tight cut of a cc2 is a 3-cut}.
This fact and \cref{3-cut contractions preserve connectivity} imply the following.

\begin{lemma}\label{core and fragments of a 3-conn cubic graph are 3-conn and cubic}
	The core as well as barrier fragments of a \cc3, with respect to any barrier, are also cubic and \mbox{$3$-connected}. \qed
\end{lemma}

We now switch our attention back to \lmity.
Each of the next three results is interesting only when the barrier $B$ is nontrivial; however, they hold even when $B$ is singleton.
The following lemma considers any vertex that does not lie in a particular barrier $B$, and relates its \lmity\ in $G$ with its \lmity\ in the corresponding $B$-fragment.

\begin{lemma}\label{outside barrier lemma}
	Let $B$ denote a nonempty barrier of a \cc3\ $G$, and let $J$ be an odd component of $G-B$. Then, a vertex $v \in V(J)$ is \lm\ in $G$ if and only if $v$ is \lm\ in the corresponding $B$-fragment $G/\overline{V(J)}$.
\end{lemma}
\begin{proof}
	Since $\partial_G(V(J))$ is a tight cut, the reverse implication follows from \cref{lm vertex in a separating cut-contraction is lm in the bigger graph}. Now suppose that $v \in V(J)$ is \lm\ in $G$, and let $M$ denote a $v$-matching.
	By \paritylemma, for each (odd) component $L$ of $G-B$, the cut $\partial_M(V(L))$ is odd. Note that:
	\[\sum_{L\,\in\,\mathcal{C}(G-B)} |\partial_M(V(L))| = \sum_{u\,\in\,B} d_M(u) = |B|\]
	Therefore, $|\partial_M(V(J))|=1$; whence, $M \cap E(G/\overline{V(J)})$ is a $v$-matching of $G/\overline{V(J)}$.
\end{proof}

The next lemma deals with vertices that belong to a particular barrier $B$, and characterizes their \lmity\ in terms of related properties of the core and of the $B$-fragments.

\begin{lemma}\label{inside barrier lemma}
	Let $B$ denote a nonempty barrier of a \cc3 $G$, and let $H[A,B]$ denote the core. Then, a vertex $b \in B$ is \lm\ in $G$ if and only if there is some vertex $a \in A$ such that $(i)$ the pair $(a,b)$ is \lm\ in~$H$ and $(ii)$ in the $B$-fragment corresponding to $a$, the contraction vertex is \lm.
\end{lemma}
\begin{proof}
	First suppose that $b$ is \lm\ in $G$, and let $M$ denote a $b$-matching of $G$.
	Note that for any (odd) component $L$ of $G-B$, by \paritylemma, $\partial_M(V(L))$ is odd.
	Also, note that:
	\[ \sum_{L\,\in\,\mathcal{C}(G-B)} |\partial_M(V(L))| = \sum_{u \in B} d_M(u) = |B|+2\]
	Consequently, for precisely one of the members of $\mathcal{C}(G-B)$, say $J$, it holds that $|\partial_M(V(J))|=3$, whereas $|\partial_M(V(L))|=1$ for every other member $L$.
	The sets $M \cap E(G/\overline{V(J)})$ and $M \cap E(H)$ demonstrate that the desired conclusion holds.

	Conversely, suppose that there exists $a \in A$ such that $(i)$ and $(ii)$ hold. Let $M_1$ denote an $(a,b)$-matching of~$H$, and let $M_2$ denote a $\lambda$-matching, of the $B$-fragment corresponding to $a$, centered at its contraction vertex. Observe that $M_1 \cup M_2$ is the desired $b$-matching of $G$.
\end{proof}

By invoking
\cref{outside barrier lemma,inside barrier lemma}, the reader may verify the following.

\begin{theorem}\label{Lambda in terms of Lambda and Rho of barriers and fragments}
	Let $B$ denote a nonempty barrier of a \cc3 $G$, let $H[A,B]$ denote the corresponding core, and let $G_1,G_2,\dots,G_{|B|}$ denote the $B$-fragments with contraction vertices $v_1,v_2,\dots,v_{|B|}$, respectively. Let $A' \subseteq A$ be defined as follows: $a \in A'$ if the contraction vertex of the corresponding $B$-fragment is \lm. Then:
	\[\ds\Lambda(G) = \bigcupdot_{i=1}^{|B|} \big(\Lambda(G_i) - v_i\big) \cupdot B'\]
	where $B' \subseteq B$ is defined as follows: $b \in B'$ if there is some vertex $a \in A'$ such that $(a,b)$ is a \lm\ pair in $H$. \qed
\end{theorem}

We now define a family $\G$ whose members satisfy $\lambda = \B$ (as proven below). Firstly, every cubic brick belongs to $\G$. Secondly, if $H[A,B]$ is any \bcc3, then every graph obtained from~$H$ by splicing with members of $\G$ at each vertex in~$A$, belongs to $\G$. For instance, the graph shown in \cref{fig: k33 spliced with 3 k4s}, that may be obtained by splicing with a copy of $K_4$ at each vertex in one color class of $K_{3,3}$, is a member of $\G$.

\begin{proposition}\label{beta lower bound 3-conn - easy direction}
	Every member $G$ of the (infinite) family $\G$ is a \mbox{$3$-connected} (nonbipartite) cubic graph that satisfies $\lambda(G)=\B(G)=n(G)$.
\end{proposition}
\begin{proof}
	The fact that each member is cubic and \mbox{$3$-connected} follows from \cref{splicing of two 3-conn cubic graphs is 3-conn}.
	We proceed by induction. If $G$ is a brick, by \cref{in a cubic brick every vertex is lm}, the desired conclusion holds.
	Otherwise, $G$ is constructed from a \bcc3\ $H[A,B]$ by splicing with members of $\G$ --- say, $G_1,G_2,\dots,G_{|A|}$ --- at each vertex of $A$.

	By the induction hypothesis, $\lambda(G_i)=\B(G_i)=n(G_i)$ for each $i \in \{1,2,\dots,|A|\}$. Observe that:
	\[n(G)=\sum_{i=1}^{|A|}n(G_i)=\sum_{i=1}^{|A|}\B(G_i)=\B(G)\]
	where the first equality is easily verified by counting, and the last one follows from \cref{beta over a barrier}.
	Note that $B$ is a barrier of $G$ and that $G_1,G_2,\dots,G_{|A|}$ are the $B$-fragments. Since the contraction vertex of every \mbox{$B$-fragment} is \lm, by \cref{each edge is lm unless it participates in a 2-cut} and \cref{Lambda in terms of Lambda and Rho of barriers and fragments}, every vertex of $G$ is \lm. In summary, $\lambda(G)=n(G)=\B(G)$.
\end{proof}

We are now ready to prove our bound $\lambda \geqslant \B$, as well as a complete characterization of the tight examples, for the case of \cc3s.

\begin{theorem}\label{beta lower bound 3-conn}
	Every \cc3\ $G$ satisfies the lower bound $\lambda(G) \geqslant \B(G)$. Furthermore, if $G$ is bipartite then $\lambda(G)=\B(G)=0$; on the other hand, if $G$ is nonbipartite, the following are equivalent:
	\begin{enumerate}[label=(\roman*)]
		\item $\lambda(G) = \B(G)$,
		\item $\B(G) = n(G)$,
		\item $G \in \G$.
	\end{enumerate}
\end{theorem}
\begin{proof}
	Note that if $G$ is bipartite, $\lambda(G)=\B(G)=0$ by \cref{lambda is 0 for bipartite graphs,mcg is bipartite iff b=0}. Now, suppose that~$G$ is nonbipartite.
	By \cref{beta lower bound 3-conn - easy direction}, $(iii)$~implies~$(i)$~and~$(ii)$.
	Clearly, $n(G) \geqslant \lambda(G)$; consequently, if we assume that the inequality $\lambda(G) \geqslant \B(G)$ holds, the reader may verify that $(ii)$~implies~$(i)$. Ergo, it suffices to prove that $\lambda(G) \geqslant \B(G)$ and that $(i)$~implies~$(iii)$. To this end, we proceed by induction on the order.

	If $G$ is a brick, by \cref{in a cubic brick every vertex is lm} and definitions of $\B(G)$ and $\G$, the desired conclusion holds.
	Now, suppose that $G$ is not a brick.
	By \cref{a 2 conn cubic graph is not theta or brick iff there is a nontrivial barrier}, $G$ has a nontrivial barrier~$B$.
	Let $H[A,B]$ denote the core of $G$,
	and let $G_1,G_2,\dots,G_{|B|}$ denote the $B$-fragments; by \cref{core and fragments of a 3-conn cubic graph are 3-conn and cubic}, all of these graphs are cubic as well as \mbox{$3$-connected}.
	By \cref{beta over a barrier} and \cref{mcg is bipartite iff b=0}, at least one of the $B$-fragments, say $G_1$, is nonbipartite.

	We define $A' \subseteq A$ as follows: $a \in A'$ if the contraction vertex of the corresponding $B$-fragment is \lm.
	We define $B' \subseteq B$ as follows: $b \in B'$ if $b$ is \lm\ in $G$.
	By \cref{outside barrier lemma}:
	\[ \ds \lambda(G) = \left(\sum_{i=1}^{|B|}\lambda(G_i)\right) - |A'| + |B'|\]
	By \cref{each edge is lm unless it participates in a 2-cut} and \cref{inside barrier lemma}, every vertex in $N_H(A')$ is \lm\ in $G$; whence, $N_H(A') \subseteq B'$.
	Consequently, $|B'| \geqslant |N_H(A')| \geqslant |A'|$.
	Also, by the induction hypothesis, $\lambda(G_i) \geqslant \B(G_i)$ holds, for each~$i~\in~\{1,2,\dots,|B|\}$.
	Hence: \[\ds \lambda(G) = \left(\sum_{i=1}^{|B|}\lambda(G_i)\right) + |B'|-|A'| \geqslant \left(\sum_{i=1}^{|B|}\B(G_i)\right) + |B'|-|A'| \geqslant \B(G)\] where the last inequality follows from the fact that $|B'| \geqslant |A'|$ and \cref{beta over a barrier}.

	Furthermore, if $\lambda(G)=\B(G)$ holds, it follows that $\lambda(G_i) = \B(G_i)$ for each $i~\in~\{1,2,\dots,|B|\}$ and $|B'|=|N_H(A')|=|A'|$.
	In particular, since $G_1$ is nonbipartite (as noted earlier), $G_1 \in \G$ by the induction hypothesis;
	by \cref{beta lower bound 3-conn - easy direction}, the contraction vertex of $G_1$ is \lm.
	Consequently, $A'$ is nonempty and since $H$ is bipartite matching covered, by \cref{characterization of bmcgs}, $A'=A$.
	In other words, the contraction vertex of each \mbox{$B$-fragment} is \lm; by \cref{lambda is 0 for bipartite graphs}, each $B$-fragment is nonbipartite and by the induction hypothesis, each of them belongs to $\G$.
	By definition, $G$ belongs to $\G$.

	This completes the proof of \cref{beta lower bound 3-conn}.
\end{proof}

We now proceed to extend this bound to all \cc2s. In order to do so, we use $2$-cuts as our induction tool (as we did in the preceding section) and the following lemma whose reverse implication is proved in \cite[Lemma 2.13 $(iii)$]{clz25}.

\begin{lemma}\label{lm vertices across a 2-cut}
	Let $G$ be a \cc2 that has a $2$-cut $C$, let $G_1$ and $G_2$ denote its marked \mbox{$C$-components}, and let $v \in V(G_1)$.
	Then, $v$ is \lm\ in $G$ if and only if $v$ is \lm\ in $G_1$.
	Consequently, $\Lambda(G) = \Lambda(G_1) \cupdot \Lambda(G_2)$, and $\lambda(G)=\lambda(G_1)+\lambda(G_2)$.
\end{lemma}
\begin{proof}
	We let $e_1$ and $e_2$ denote the marker edges of $G_1$ and $G_2$, respectively.
	First, suppose that $v$ is \lm~in $G$, and let $M$ denote a \mbox{$v$-matching}.
	By \paritylemma, $|C \cap M|$ is even.
	If $C\cap M = \emptyset$, let $M_1:=M \cap E(G_1)$;
	otherwise, let $M_1:=(M \cap E(G_1)) + e_1$.
	In either case, the reader may verify that $M_1$ is a $v$-matching of $G_1$.
	Conversely, suppose that $v$ is \lm~in $G_1$, and let $M_1$ denote a \mbox{$v$-matching}.
	Let $M_2$~and~$N_2$ denote perfect matchings of $G_2$ containing $e_2$ and not containing~$e_2$, respectively.
	If $e_1 \in M_1$, let $M:=(M_1-e_1) + (M_2-e_2) + C$;
	otherwise, let $M:=M_1 + N_2$.
	In either case, the reader may verify that $M$ is a $v$-matching of~$G$.
\end{proof}

We now define a family $\G'$ of \cc2s whose members satisfy $\lambda = \B$ as follows: $\G'$~comprises those graphs each of whose \mbox{$3$-connected} pieces is either bipartite or belongs to~$\G$.
For instance, the graph shown in \cref{fig: k4 glue k33} is a member of $\G'$.
The following is easily proved by invoking \cref{beta lower bound 3-conn}, \cref{bstar across even 2-cut}, \cref{lm vertices across a 2-cut} and \cref{theta thetabar and n-nonbip over a 2 cut}.

\begin{corollary}
	Every \cc2~$G$ satisfies the lower bound $\lambda(G) \geqslant \B(G)$; furthermore, the following are equivalent:
	\begin{enumerate}[label=(\roman*)]
		\item $\lambda(G) = \B(G)$,
		\item $\B(G) = n_{\sf nonbip}(G)$,
		\item $G \in \G'$. \qed
	\end{enumerate}
\end{corollary}

We now redirect our attention towards characterizing \cc2s that satisfy $\lambda=n$.
As usual, we handle \cc3s first and then use $2$-cuts as an induction tool.
It turns out that in order to describe \cc3s that satisfy $\lambda=n$, we also need to describe those that satisfy $\lambda=n-1$.
Interestingly, we find it easier to characterize them together, rather than characterizing them one by one.

To this end, we recursively describe a family $\N$ of ordered pairs; the first element of each member is a graph~$G$, and the second member is a subset of~$V(G)$.
Firstly, $(G,V(G)) \in \N$ for every cubic brick~$G$.
Secondly, if $H[A,B]$ is any \bcc3 that is distinct from $\Theta$ \mbox{where $A:=\{a_1,a_2,\dots,a_{|B|}\}$;} $A_0$ is any subset of $A$ such that for each $b \in B$, there exists an $a \in A_0$ so that~$(a,b)$ is a \lm\ pair in $H$; and $A_1$ is any subset of $A-A_0$ of cardinality at most one; then the pair $(G,X)$ defined below is also in $\N$.
\begin{enumerate}[label=(\roman*)]
	\item $G$ is obtained from $H$ by splicing at each $a_i \in A_0$ with any graph $G_i$ such that $(G_i,V(G_i)) \in \N$, and splicing at each $a_i \in A-A_0-A_1$ with any graph $G_i$ at the vertex $v_i$ such that $(G_i, \{v_i\}) \in \N$.
	\item If $A_1$ is empty, $X:=V(G)$; otherwise, $X:=A_1$.
\end{enumerate}

Observe that, for each member $(G,X) \in \N$ defined above recursively, the set $B$ is a nontrivial barrier of~$G$; consequently, $G$ is not a brick. Furthermore, if $A_1$ is empty then each component of~$G-B$ is nontrivial, whereas if $A_1$ is singleton then $G-B$ has precisely one trivial component whose vertex set is $A_1$.
For instance, the ordered pair comprising the graph shown in \cref{fig: k33 spliced with 2 k4s}, and its vertex subset~$\{v\}$, is a member of $\N$.

\fig{A \cc3 that satisfies $\lambda=n-1$}{fig: k33 spliced with 2 k4s}{\begin{tikzpicture}[scale=1.35]

\tikzmath{
	\height = 0.7;
	\halfsidelength = sin(30)*\height;
	\sidelength = 2*\halfsidelength;
	\offset = 0;
	\y = -0.7;
}


\foreach \i in {1,2}{
	\node[vtx] (b\i) at (1.25*\i,0.825) {};
	\tikzmath{ \x = 1.25*\i+\offset*(\i-2); }
	\node[vtx] (a\i2) at (\x,\y+2*\height/3) {};
	\node[vtx] (a\i1) at (\x-\halfsidelength, \y-\height/3) {};
	\node[vtx] (a\i3) at (\x+\halfsidelength, \y-\height/3) {};
}

\node[vtx] (b3) at (1.25*3,0.825) {};
\node[vtx] (a3) at (1.25*3,\y) {};


\foreach \i in {1,2}{
	\draw (b2) -- (a\i2);
	\draw (a\i1) -- (a\i2) -- (a\i3) -- (a\i1);
}
\draw (a21) -- (b1);
\draw (a23) -- (b3);

\draw (a3) -- (b2);
\draw (a3) -- (b3);

\draw plot [smooth,tension=1.1] coordinates {(a13) ($(a22)!0.35!(b2)$) (b3)};
\draw plot [smooth,tension=1.1] coordinates {(a3) ($(a22)!0.35!(b2)$) (b1)};
\draw plot [smooth,tension=1.1] coordinates {(a11) ($(a11)!0.525!(b1)+(-0.225,0)$) (b1)};

\node at ($(a3)+(0.2,0)$) {$v$};

\end{tikzpicture}}

The ordered pair comprising the graph $G$ shown in \cref{fig: lambda > beta example}, and its vertex set $V(G)$, is a member of~$\N$.
Its membership may be explained as follows. Let $H[A,B]:=K_{3,3}$ where $A:=\{a_1,a_2,a_3\}$, let~$A_0:=\{a_1,a_2\}$, let both $G_1$ and $G_2$ be $K_4$, and let $G_3$ be the graph shown in \cref{fig: k33 spliced with 2 k4s} with $v$ as its vertex where splicing is done.
We remark that $G$ may also be obtained by splicing two copies of $G_3$ at their respective copies of $v$, although this viewpoint does not explain membership in $\N$.
We now argue that each graph that appears in $\N$ satisfies $\lambda \in \{n-1,n\}$.

\begin{proposition}\label{characterization of cc3s with lambda at least n-1 - easy direction}
	Each member $(G,X)$ of the family $\N$ is a \cc3 that satisfies the following: if $X=V(G)$ then $\Lambda(G)=V(G)$ whereas if $|X|=1$ then $\Lambda(G)=V(G)-X$.
\end{proposition}
\begin{proof}
	We proceed by induction on the order.
	If $G$ is a brick, then $X=V(G)$ and $\Lambda(G)=V(G)$; we are done.
	Now suppose that $G$ is not a brick.

	Ergo, $G$ is constructed from a \bcc3 $H[A,B]$ of order six or more, subsets $A_0$ and $A_1$ of $A$, and smaller members $(G_i,X_i) \in \N$, for each $a_i \in A-A_1$, as described in the definition of~$\N$.
	By \cref{splicing of two 3-conn cubic graphs is 3-conn}, $G$ is a \cc3.
	Note that $A_1$ is either singleton or empty; in the former case, by \cref{characterization of non lm vertices}, the only member of $A_1$ is not \lm\ in~$G$.
	Consequently, in both cases, it suffices to argue that each member of $V(G)-A_1$ is \lm\ in~$G$; see condition $(ii)$ in the definition of $\N$.

	By condition $(i)$ in definition of $\N$ and by the induction hypothesis, we infer that:
	$\Lambda(G_i)=V(G_i)$ for each $a_i \in A_0$; whereas $\Lambda(G_i)=V(G_i)-v_i$ for each $a_i \in A-A_0-A_1$.
	We now invoke \cref{outside barrier lemma} to deduce that each member of $V(G)-B-A_1$ is \lm\ in $G$.

	Now, let $b \in B$. By definition of $A_0$, there exists an $a_i \in A_0$ such that $(a_i,b)$ is \lm\ in~$H$. Since $v_i$ is \lm\ in $G_i$ as noted above, by \cref{inside barrier lemma}, $b$ is \lm\ in $G$. Thus, each member of $B$ is \lm\ in $G$.
\end{proof}

We proceed to prove the converse, and thus establish our characterization of \cc3s that satisfy $\lambda \in \{n-1,n\}$.

\begin{theorem}\label{characterization of cc3s with lambda at least n-1}
	For every \cc3 $G$:
	\begin{enumerate}[label=(\roman*)]
		\item $\lambda(G)=n(G)$ if and only if $(G,V(G)) \in \N$, and
		\item $\lambda(G)=n(G)-1$ if and only if $(G,\{u\}) \in \N$, where $u$ is its unique vertex that is not \lm.
	\end{enumerate}
\end{theorem}
\begin{proof}
	By \cref{characterization of cc3s with lambda at least n-1 - easy direction}, it remains to prove the forward implication of each statement.
	To this end, let $G$ be a \cc3 with $\lambda(G) \geqslant n(G)-1$; clearly, $G$ is not $\Theta$.
	We proceed by induction on the order.
	If $G$ is a brick, then $\lambda(G)=n(G)$ and the desired conclusion holds by definition of $\N$.
	Now suppose that $G$ is not a brick.

	If $\lambda(G) = n(G)-1$, we let $u$ denote the only vertex that is not \lm, and by invoking \cref{characterization of non lm vertices}, we let $B$ denote a (nontrivial) barrier such that $u$ is isolated in $G-B$.
	Otherwise, by invoking \cref{a 2 conn cubic graph is not theta or brick iff there is a nontrivial barrier}, we let $B$ denote any nontrivial barrier.
	Let $H[A,B]$ denote the corresponding core with $A:=\{a_1,a_2,\dots,a_{|B|}\}$, let $G_i$ denote the \mbox{$B$-fragment} corresponding to $a_i \in A$, and let $v_i$ denote the contraction vertex of $G_i$.
	By \cref{core and fragments of a 3-conn cubic graph are 3-conn and cubic}, $H$~is a \bcc3 and each~$G_i$ is a \cc3.

	Since each vertex except $u$ (if it exists) is \lm\ in $G$, by \cref{outside barrier lemma}, for each $B$-fragment $G_i$, except for the one corresponding to $u$ (if it exists), $V(G_i)-v_i \subseteq \Lambda(G_i)$; in particular, $\lambda(G_i) \geqslant n(G_i)-1$.
	Now, we define a subset $A_0$ of $A$ as follows: $a_i \in A_0$ if the contraction vertex $v_i$ is \lm\ in $G_i$.
	If $u$ exists, we let $A_1:=\{u\}$; otherwise, let $A_1:= \emptyset$.
	Note that $\lambda(G_i)=n(G_i)$ for each $a_i \in A_0$; consequently, by the induction hypothesis, $(G_i,V(G_i)) \in \N$.
	On the other hand, $\Lambda(G_i) = V(G_i)-v_i$ for each $a_i \in A-A_0-A_1$; consequently, by the induction hypothesis, $(G_i,\{v_i\}) \in \N$. Note that condition $(i)$ in the definition of $\N$ is satisfied.

	Now, we shall argue that our choices of $A_0$ and $A_1$ satisfy the conditions stated in the definition of $\N$.
	For each $b \in B$, since $b$ is \lm\ in $G$, by \cref{inside barrier lemma}, there exists an $a_i \in A$ such that $(a_i,b)$ is \lm\ in $H$ and the contraction vertex $v_i$ is \lm\ in $G_i$.
	Thus, it follows from the definition of~$A_0$ in the preceding paragraph that for each $b \in B$, there exists $a_i \in A_0$ such that $(a_i,b)$ is \lm\ in~$H$.
	Also, observe that $u$ (if it exists) is not in $A_0$; consequently, $A_1$ is a subset of $A-A_0$ of cardinality at most one.

	In order to satisfy condition $(ii)$ in the definition of $\N$, we define $X$ as follows. If $\lambda(G)=n(G)$, we let $X:=V(G)$; otherwise let $X:=\{u\}$. It follows that $(G,X) \in \N$ as desired.
\end{proof}

In order to characterize all \cc2s that satisfy $\lambda=n$,
we define the family $\N'$ as follows. A \cc2 $G'$ belongs to $\N'$ if the pair $(G,V(G))$ belongs to $\N$ for each \con3 piece $G$ of $G'$.
An example is shown in \cref{fig: not 3 connected lambda = n example}.

\fig{A member of $\mathcal{N}'$}{fig: not 3 connected lambda = n example}{\begin{tikzpicture}[scale=1.2]
	\begin{scope}[rotate=90]
		\tikzmath{
				\sidelength = 0.7;
				\halfsidelength = \sidelength/2;
				\height = sin(60)*\sidelength;
				\offset = 0.65;
				\xcoord = 1;
				\ycoord = 0.4;
			}
		\coordinate (b11) at (-\offset + \xcoord-\halfsidelength,\ycoord-\height);
		\coordinate (b12) at (-\offset + \xcoord,\ycoord);
		\coordinate (b13) at (-\offset + \xcoord+\halfsidelength,\ycoord-\height);
	
		\coordinate (b21) at (\offset + \xcoord-\halfsidelength,\ycoord-\height);
		\coordinate (b22) at (\offset + \xcoord,\ycoord);
		\coordinate (b23) at (\offset + \xcoord+\halfsidelength,\ycoord-\height);
	
		\coordinate (a1) at (1 - 1, 1);
		\coordinate (a2) at (1,1);
		\coordinate (a3) at (1 + 1, 1);
	
		\draw (b11) -- (b12) -- (b13) -- (b11);
		\draw (b21) -- (b22) -- (b23) -- (b21);
		\draw (b11) -- (a1);
		\draw (b12) -- (a2);
		\draw (b13) -- (a3);
		\draw (b21) -- (a1);
		\draw (b22) -- (a2);
		\draw (b23) -- (a3);
	
		\begin{scope}[yshift=74, yscale=-1]
			\tikzmath{
				\sidelength = 0.7;
				\halfsidelength = \sidelength/2;
				\height = sin(60)*\sidelength;
				\offset = 0.65;
				\xcoord = 1;
				\ycoord = 0.4;
			}
			\coordinate (2b11) at (-\offset + \xcoord-\halfsidelength,\ycoord-\height);
			\coordinate (2b12) at (-\offset + \xcoord,\ycoord);
			\coordinate (2b13) at (-\offset + \xcoord+\halfsidelength,\ycoord-\height);
	
			\coordinate (2b21) at (\offset + \xcoord-\halfsidelength,\ycoord-\height);
			\coordinate (2b22) at (\offset + \xcoord,\ycoord);
			\coordinate (2b23) at (\offset + \xcoord+\halfsidelength,\ycoord-\height);
	
			\coordinate (2a1) at (1 - 1, 1);
			\coordinate (2a2) at (1,1);
			\coordinate (2a3) at (1 + 1, 1);
	
			\draw (2b11) -- (2b12) -- (2b13) -- (2b11);
			\draw (2b21) -- (2b22) -- (2b23) -- (2b21);
			\draw (2b11) -- (2a1);
			\draw (2b12) -- (2a2);
			\draw (2b13) -- (2a3);
			\draw (2b21) -- (2a1);
			\draw (2b22) -- (2a2);
			\draw (2b23) -- (2a3);
		\end{scope}
	
		\draw (a2) -- (2a2);
		\draw (a3) -- (2a3);
	
		\draw (a1)node[vtx]{} (a2)node[vtx]{} (a3)node[vtx]{} (b11)node[vtx]{} (b12)node[vtx]{} (b13)node[vtx]{} (b21)node[vtx]{} (b22)node[vtx]{} (b23)node[vtx]{};
	
		\draw (2a1)node[vtx]{} (2a2)node[vtx]{} (2a3)node[vtx]{} (2b11)node[vtx]{} (2b12)node[vtx]{} (2b13)node[vtx]{} (2b21)node[vtx]{} (2b22)node[vtx]{} (2b23)node[vtx]{};
	\end{scope}

	\begin{scope}[shift={(-1.3,-1.5)}, yscale=-1]
		\tikzmath{
			\height = 1;
			\halfsidelength = sin(30)*\height;
			\sidelength = 2*\halfsidelength;
		}

		\node[vtx] (v1) at (0,0) {};
		\node[vtx] (v2) at (-\halfsidelength, -\height) {};
		\node[vtx] (v3) at ( \halfsidelength, -\height) {};
		\node[vtx] (v4) at (0, -2*\height/3) {};

		\draw (v1) -- (v2) -- (v4) -- (v3) -- (v1);
		\draw (v4) -- (v1);
	\end{scope}

	\draw (v2) -- (2a1);
	\draw (v3) -- (a1);
\end{tikzpicture}}

The following is an immediate consequence of \cref{characterization of cc3s with lambda at least n-1} and \cref{lm vertices across a 2-cut}.

\begin{corollary}
A \cc2 $G$ satisfies $\lambda(G)=n(G)$ if and only if $G \in \N'$. \qed
\end{corollary}

\section*{Acknowledgements}
We are very grateful to both anonymous reviewers for their careful reading and helpful suggestions, including spotting a missing case in our original proof of \cref{characterization of pairs that are not lm}.

\end{document}